\documentclass{amsart}
\usepackage[utf8]{inputenc}

\usepackage{amsmath}
\usepackage{amssymb}
\usepackage{amsthm}
\usepackage[margin=3cm]{geometry}
\usepackage{enumerate}
\usepackage[OT1]{fontenc}
\usepackage[all,cmtip]{xy}
\usepackage[english]{babel}
\usepackage{tikz}
\usepackage{tikz-cd}

\newtheorem{theorem}{Theorem}[section]
\newtheorem{definition}[theorem]{Definition}
\newtheorem{corollary}[theorem]{Corollary}
\newtheorem{lemma}[theorem]{Lemma}
\newtheorem{prop}[theorem]{Proposition}

\newtheorem{remark}{Remark}

\theoremstyle{definition}
\newtheorem{example}[theorem]{Example}

\newcommand{\spn}{\text{span}}
\newcommand{\bR}{\mathbb R}

\newcommand{\bC}{\mathbb C}

\newcommand{\bZ}{\mathbb Z}
\newcommand{\bN}{\mathbb N}
\newcommand{\bT}{\mathbb T}
\newcommand{\bD}{\mathbb D}

\newcommand{\C}{\mathcal C}

\renewcommand{\H}{\mathcal H}

\newcommand{\K}{\mathcal K}

\newcommand{\M}{\mathcal M}

\newcommand{\T}{\mathcal T}

\newcommand{\AND}{\text{ and }}

\renewcommand{\to}{\rightarrow}
\newcommand{\ol}{\overline}
\newcommand{\id}{\text{id}}
\newcommand{\diag}{\text{diag}}
\makeatletter
\@namedef{subjclassname@2020}{%
  \textup{2020} Mathematics Subject Classification}
\makeatother

\title{Relative $K$-theory for $C^*$-algebras}

\author{Mitch Haslehurst}
\email{mjthasle@gmail.com}


\begin{document}

\begin{abstract}
Given C$^*$-algebras $A$ and $B$ and a $^*$-homomorphism $\phi:A\to B$, we adopt the portrait of the relative $K$-theory $K_*(\phi)$ due to Karoubi using Banach categories and Banach functors. We show that the elements of the relative groups may be represented in a simple form. We prove the existence of two six-term exact sequences, and we use these sequences to deduce the fact that the relative theory is isomorphic, in a natural way, to the $K$-theory of the mapping cone.
\end{abstract}

\maketitle

\section{Introduction}
\label{intro}

In this paper we develop a portrait of relative $K$-theory for C$^*$-algebras by following an approach due to Karoubi that uses Banach categories and Banach functors. For every $^*$-homomorphism $\phi:A\to B$ between C$^*$-algebras $A$ and $B$, we produce two abelian groups $K_0(\phi)$ and $K_1(\phi)$ that give information about how the $K$-theory of $A$ and $B$ are related through $\phi$. In fact, the assignments $\phi\mapsto K_0(\phi)$ and $\phi\mapsto K_1(\phi)$ may be regarded as functors in a natural way, and through this we obtain a homology theory that satisfies Bott periodicity (the long exact sequence being the one in part (ii) of Theorem \ref{main}), and a tight connection between the $K$-theory of $A$ and $B$, the induced maps $\phi_*:K_*(A)\to K_*(B)$, and $K_*(\phi)$ is contained in the six-term exact sequence in part (i) of Theorem \ref{main}. 
	
The notion of relative $K$-theory is not new in the subject of operator algebras. Indeed, references such as \cite{blackadar} and \cite{higsonroe} contain a concise exposition under the assumptions that $A$ is unital, $B$ is a quotient of $A$ by some closed, two-sided ideal $I$, and $\phi$ is the quotient map. The relative group produced is denoted $K_0(A,A/I)$ (although the notation varies throughout the literature). These assumptions are quite reasonable, since $K_0(A,A/I)$ provides the noncommutative generalization of the relative group $K^0(X,Y)$ in topological $K$-theory, where $X$ is a compact space and $Y$ is a closed subset of $X$. The key feature of these groups in both cases, commutative or not, is that they satisfiy excision: they depend only on a smaller substructure in question, namely $X-Y$ in the topological case and $I$ in the noncommutative case. Specifically, the group $K^0(X,Y)$ is isomorphic to the group $K^0(X/Y,\{y\})$, where $\{y\}$ is the set $Y$ collapsed to a point, and the group $K_0(A,A/I)$ is isomorphic to the group $K_0(\tilde I, \bC)$, where $\tilde I$ is the unitization of $I$ and we identify $\tilde I/I$ with $\bC$.
	
To obtain a relative theory for a more general $^*$-homomorphism, we appeal to a construction of Karoubi in \cite{karoubi}. The approach is to describe the elements of the relative groups using triples consisting of two objects and a morphism from a Banach category. This generalizes the classical method of producing $K$-groups via vector bundles as seen in Atiyah's seminal work \cite{atiyah}: $K^0(X,Y)$ may be constructed via triples of the form $(E,F,\alpha)$, where $E$ and $F$ are vector bundles over $X$ and $\alpha:E|_Y\to F|_Y$ is an isomorphism between the bundles $E$ and $F$ when restricted to $Y$. This approach may appear somewhat basic because it is at odds with the mapping cone, a shortcut seen in both topological and operator $K$-theory. Indeed, analogous versions of the six-term exact sequences in the main result here (Theorem \ref{main}) are obtained very easily using standard methods if one uses $K_*(C_\phi)$ as the definition of relative $K$-theory, where $C_\phi$ denotes the mapping cone of $\phi:A\to B$. Moreover, it turns out that the two portraits are isomorphic in a natural way (part (iv) of Theorem \ref{main}). It is therefore reasonable to ask why one would employ an alternative portrait at all.
	
First, there is much more freedom in selecting the elements from the algebras that represent elements of the relative groups via the setup in \cite{karoubi}, which makes viewing and working with the groups easier in many situations. Second, certain maps in the six-term exact sequences are easier to compute. All in all, the resulting presentation is much simpler to work with, and applications of a simpler portrait in the context of C$^*$-algebras are beginning to make their way into the literature. Indeed, the proof of the excision theorem in \cite{putnam3}, which is a significant generalization of its predecessor in \cite{putnam2} (where the mapping cone was used), rests heavily on the new presentation, particularly on the notions of "isomorphism" of triples and "elementary" triples, to address certain fine, technical details. A portrait using partial isometries is developed in \cite{putnam2}, and in fact, the map $\kappa$ constructed there is essentially the same as $\Delta_\phi$ in part (iv) of Theorem \ref{main}, the difference being that $\Delta_\phi$ is a functorially induced group isomorphism, while $\kappa$ is a bijective map constructed concretely. Also worth mentioning is recent work on groupoid homology \cite{gabe}, where a portrait of the relative $K_0$-group using triples is presented in order to elucidate the connections between $K$-theory and homology. An isomorphism between the resulting monoid and the $K_0$-group of the mapping cone is constructed under the assumption that $A$ and $B$ are unital and $\phi$ is unit-preserving. The unital assumption can be done away with if $\phi$ is nondegenerate and $A$ contains an approximate identity of projections. Here, the general description of the relative groups allows us to do away with these assumptions. To provide further justification, we elaborate on the $K_0$- and $K_1$-groups separately.
	
It is a standard fact that the three usual notions of equivalence of projections, Murray-von Neumann, unitary, and homotopy, are all stably equivalent; in other words, they are the same modulo passing to matrix algebras. In constructing a relative $K_0$-group, it is therefore necessary to select a notion of equivalence with which to build the elements. The mapping cone $C_\phi$ is made from paths of projections in $B$ with one endpoint equal to a scalar projection and the other endpoint in the image of $\phi$. Therefore, in effect, $K_0(C_\phi)$ catalogues projections arising from $A$ that are homotopic when moved to $B$ via $\phi$. It is often more desirable to describe equivalences of projections using partial isometries (see Example \ref{diag}), from which the newer portrait is built. It is also worth mentioning that, although both portraits often require unitizations of the algebras involved, the appended unit seems to be less of a hindrance in the newer portrait.
	
As for $K_1$, homotopy is a much more natural equivalence and hence the portrait of the relative $K_1$ group gets less of a makeover than that of $K_0$. In fact, the two portraits are more or less the same. However, we draw a useful property from \cite{karoubi} which deserves to be mentioned. It is possible to define (ordinary and relative) $K_1$-groups more generally using partial unitaries (elements which are partial isometries and normal) rather than unitaries alone. This more general representation of group elements is especially convenient if one of (or both) $A$ and $B$ are not unital but contain nontrivial projections, such as $\K$, the compact operators on a separable Hilbert space (see Example \ref{gamma}).
	
The goal of the paper is to develop a clear picture of relative $K$-theory for C$^*$-algebras using the setup in \cite{karoubi}, as an alternative to the mapping cone. We remark that a preliminary development of the picture may be found in \cite{putnam3}, where the relative $K_0$-group of an inclusion $A'\subseteq A$ is described using this approach. We also remark that, although the intention in \cite{karoubi} is mainly to develop topological $K$-theory, the setup lends itself quite well to C$^*$-algebras.
	
The paper is organized as follows. In section 2 we state the theorems and discuss some examples. In section 3 we present the definition of relative $K$-theory in the context of C$^*$-algebras and show that the elements of the relative groups can be represented in a simple form. In section 4 we prove the results.

\section*{\normalsize Acknowledgements}

The content of this paper constitutes a portion of the research conducted for my PhD dissertation. I am grateful to my advisor Ian F. Putnam for suggesting the project to me, and for numerous helpful discussions. I also send my thanks to the referee for their careful reading and useful suggestions.

\section{Summary and examples}

Before we begin, we state a rather important remark regarding notation.
 
\begin{remark}
 In \cite{rordam}, the symbols $K_0(\phi)$ and $K_1(\phi)$ are used to denote the group homomorphisms $K_0(A)\to K_0(B)$ and $K_1(A)\to K_1(B)$ induced by $\phi$. Throughout this paper, the symbols $K_0(\phi)$ and $K_1(\phi)$ will be used to denote the relative groups, not group homomorphisms. Induced maps will instead be denoted more classically as $\phi_*$.
\end{remark}

\begin{definition}
We define \textbf{C$^*$-hom} to be the following category. The objects of \textbf{C$^*$-hom} are $^*$-homomorphisms $\phi:A\to B$, where $A$ and $B$ are C$^*$-algebras. A morphism from the $^*$-homomorphism $\phi:A\to B$ to the $^*$-homomorphism $\psi:C\to D$ is a pair $(\alpha,\beta)$ of $^*$-homomorphisms $\alpha:A\to C$ and $\beta:B\to D$ such that the diagram
\begin{equation}\label{comm}
\begin{tikzcd}
A \arrow[r, "\alpha"] \arrow[d, "\phi"] &   C \arrow[d, "\psi"]              \\
B \arrow[r, "\beta"]              & D           
\end{tikzcd}
\end{equation}
commutes. If $(\alpha,\beta)$ is a morphism from $\phi:A\to B$ to $\psi:C\to D$ and $(\gamma,\delta)$ is a morphism from $\psi:C\to D$ to $\eta:E\to F$, their composition is $(\gamma\circ\alpha,\delta\circ\beta)$. The identity morphism from $\phi:A\to B$ to itself is \emph{$(\id_A,\id_B)$}, where \emph{$\id_A:A\to A$} is the identity map, \emph{$\id_A(a)=a$} for all $a$ in $A$.
\end{definition}

As the reader can easily check, the class of $^*$-homomorphisms that map to the zero C$^*$-algebra, $\phi:A\to\{0\}$, forms a subcategory of \textbf{C$^*$-hom} that is isomorphic to the category of C$^*$-algebras in an obvious way. If one restricts to this subcategory, the usual definition of complex $K$-theory $K_*(A)$ for C$^*$-algebras is recovered through the construction of the relative theory $K_*(\phi)$.

We outline a simplified picture of the relative groups $K_0(\phi)$ and $K_1(\phi)$ to help understand the results and examples. Full definitions will be given in section 3.
 
The group $K_0(\phi)$ is made from triples $(p,q,v)$, where $p$ and $q$ are projections in some matrix algebras over $\tilde A$, the unitization of $A$, and $v$ is an element in a matrix algebra over $\tilde B$ such that $v^*v=\phi(p)$ and $vv^*=\phi(q)$ (if $A$ and $B$ are unital and $\phi(1)=1$, we may ignore unitizations, see Proposition \ref{unit}). The triples are sorted into equivalence classes, denoted $[p,q,v]$, and are given a well-defined group operation by the usual block diagonal sum,
\[[p,q,v]+[p',q',v']=[p\oplus p',q\oplus q',v\oplus v']\]
For two triples $(p,q,v)$ and $(p',q',v')$ to yield the same equivalence class, $p$ and $p'$ must be (at least stably) Murray-von Neumann equivalent in $\tilde A$, as must be $q$ and $q'$. Moreover, elements $c$ and $d$ implementing such equivalences must play well with $v$ and $v'$ in that we require $\phi(d)v=v'\phi(c)$.
 
$K_1(\phi)$ is made from triples $(p,u,g)$, where $p$ is a projection in $M_\infty(\tilde A)$, $u$ is a unitary in $pM_\infty(\tilde A)p$, and $g$ is a unitary in $C([0,1])\otimes\phi(p)M_\infty(\tilde B)\phi(p)$ such that $g(0)=\phi(p)$ and $g(1)=\phi(u)$. The triples are sorted into equivalence classes, denoted $[p,u,g]$, and are given a well-defined group operation by diagonal sum as before, although we have the formula
\[[p,u,g]+[p',u',g']=[p,uu',gg']\]
if $p=p'$. For two triples $(p,u,g)$ and $(p',u',g')$ to yield the same equivalence class, $p$ and $p'$ must be (at least stably) Murray-von Neumann equivalent in $\tilde A$, and a partial isometry $v$ implementing such an equivalence must satisfy $vu=u'v$ and $\phi(v)g(s)=g'(s)\phi(v)$ for $0\leq s\leq1$. The equivalence may also be described as stable homotopy: $u$ and $u'$ must be (at least stably) homotopic, as must be $g$ and $g'$. Moreover, such homotopies $u_t$ and $g_t$ must satisfy $g_t(1)=\phi(u_t)$ for $0\leq t\leq1$.
    
The assignments $\phi\mapsto K_0(\phi)$ and $\phi\mapsto K_1(\phi)$ are functors, as follows (see Proposition \ref{induce} for a proof). If we have the commutative diagram (\ref{comm}), that is, a morphism $(\alpha,\beta)$ from $\phi$ to $\psi$, then there are group homomorphisms $(\alpha,\beta)_*:K_j(\phi)\to K_j(\psi)$ for $j=0,1$ that satisfy $(\alpha,\beta)_*([p,q,v])=[\alpha(p),\alpha(q),\beta(v)]$ and $(\alpha,\beta)_*([p,u,g])=[\alpha(p),\alpha(u),\beta(g)]$.

\begin{theorem}\label{main}
The constructions in section 3 produce, for every integer $n\geq0$, a functor $K_n$ from the category \textbf{C$^*$-hom} to the category of abelian groups that satisfies the following properties.
\begin{enumerate}[(i)]
    \item If $\phi:A\to B$ is a $^*$-homomorphism of C$^*$-algebras $A$ and $B$, the groups $K_0(\phi)$ and $K_1(\phi)$ fit into the six-term exact sequence
    \begin{center}
	    \begin{tikzcd}
K_1(B) \arrow[rr, "\mu_0"] &  & K_0(\phi) \arrow[rr, "\nu_0"]  &  & K_0(A) \arrow[dd, "\phi_*"]      \\
                                &  &                                  &  &                                  \\
K_1(A) \arrow[uu, "\phi_*"]     &  & K_1(\phi) \arrow[ll, "\nu_1"'] &  & K_0(B) \arrow[ll, "\mu_1"']
\end{tikzcd}
	\end{center}
	The maps $\nu_0$ and $\nu_1$ are given by the formulas
    \[\nu_0([p,q,v])=[p]-[q]\qquad\nu_1([p,u,g])=[u+1_n-p]\]
    where $p$ and $q$ are projections in $M_n(\tilde A)$. The maps $\mu_0$ and $\mu_1$ are given by the formulas
    \[\mu_0([u])=[1_n,1_n,u]\qquad\mu_1([p]-[q])=[1_n,1_n,f_pf_q^*]\]
    where $u$ is a unitary in $M_n(\tilde B)$, and $f_p(t)=e^{2\pi itp}$ for a projection $p$ in $M_n(\tilde B)$. If $\phi(a)=0$ for all $a$ in $A$, then the sequence splits at $K_0(A)$ and $K_1(A)$, i.e., both $\nu_0$ and $\nu_1$ have a right inverse.
	\item If
    \begin{center}
\begin{tikzcd}
0 \arrow[r] & I \arrow[r, "\iota_A"] \arrow[d, "\psi"] & A \arrow[r, "\pi_A"] \arrow[d, "\phi"] & A/I \arrow[r] \arrow[d, "\gamma"] & 0 \\
0 \arrow[r] & J \arrow[r, "\iota_B"]                     & B \arrow[r, "\pi_B"]                   & B/J \arrow[r]                    & 0
\end{tikzcd}
\end{center}
is a commutative diagram with exact rows, then for each integer $n\geq1$, there is a natural connecting map $\partial_n:K_n(\gamma)\to K_{n-1}(\psi)$ such that the sequence\small
    \begin{center}
\begin{tikzcd}
\cdots \arrow[r, "\pi_*"] & K_2(\gamma) \arrow[r, "\partial_2"] & K_1(\psi) \arrow[r, "\iota_*"] & K_1(\phi) \arrow[r, "\pi_*"] & K_1(\gamma) \arrow[r, "\partial_1"] & K_0(\psi) \arrow[r, "\iota_*"] & K_0(\phi) \arrow[r, "\pi_*"] & K_0(\gamma)
\end{tikzcd}
    \end{center}\normalsize
    is exact.
    \item The theory satisfies Bott periodicity. Specifically, for each $^*$-homomorphism and each integer $n\geq0$, there is an isomorphism $\beta_\phi:K_n(\phi)\to K_{n+2}(\phi)$ that is natural in the sense that if the diagram (\ref{comm}) is commutative, then the diagram
    \begin{center}
\begin{tikzcd}
K_n(\phi) \arrow[dd, "\beta_\phi"] \arrow[rr, "{(\alpha,\beta)_*}"] &  & K_n(\psi) \arrow[dd, "\beta_\psi"] \\
                                                                    &  &                                    \\
K_{n+2}(\phi) \arrow[rr, "{(\alpha,\beta)_*}"]                 &  & K_{n+2}(\psi)                      
\end{tikzcd}
    \end{center}
    is commutative. It follows that the long exact sequence in part (ii) collapses to a six-term exact sequence
    \begin{center}
\begin{tikzcd}
K_0(\psi) \arrow[rr, "\iota_*"] &  & K_0(\phi) \arrow[rr, "\pi_*"] &  & K_0(\gamma) \arrow[dd, "\partial_0"]  \\
                       &  &                      &  &                        \\
K_1(\gamma) \arrow[uu, "\partial_1"]  &  & K_1(\phi) \arrow[ll, "\pi_*"] &  & K_1(\psi) \arrow[ll, "\iota_*"]
\end{tikzcd}
	\end{center}
	where the map $\partial_0:K_0(\gamma)\to K_1(\psi)$ is the composition $\partial_2\circ\beta_\gamma$.
	\item\label{test} If $\phi:A\to B$ is a $^*$-homomorphism of C$^*$-algebras $A$ and $B$, there are isomorphisms $\Delta_\phi:K_*(\phi)\to K_*(C_\phi)$ that are natural in the sense that if (\ref{comm}) is commutative, then the diagram
    \begin{center}
\begin{tikzcd}
K_j(\phi) \arrow[dd, "\Delta_\phi"] \arrow[rr, "{(\alpha,\beta)_*}"] &  & K_j(\psi) \arrow[dd, "\Delta_\psi"] \\
                                                        &  &                                  \\
K_j(C_\phi) \arrow[rr, "(\alpha\oplus C\beta)_*"]                      &  & K_j(C_\psi)                     
\end{tikzcd}
    \end{center}
    is commutative, where $C_\phi$ is the mapping cone of $\phi$.
\end{enumerate}
\end{theorem}

Regarding part (iv), the conclusion actually implies that the isomorphisms implement an invertible natural transformation between the functors $\phi\mapsto K_*(\phi)$ and $\phi\mapsto K_*(C_\phi)$ from the category \textbf{C$^*$-hom} to the category of abelian groups. When we speak of the transformation, we will denote it simply by $\Delta$.

We also collect some properties of the homology theory that are analogues of properties of C$^*$-algebra $K$-theory.

\begin{theorem}\label{axiom}
The homology theory $(K_n)_{n\geq0}$ on \textbf{C$^*$-hom} has the following properties.
\begin{enumerate}[(i)]
    \item Homotopy invariance: suppose that $\alpha_t:A\to C$ and $\beta_t:B\to D$ are $^*$-homomorphisms for every $0\leq t\leq1$, the maps $t\mapsto\alpha_t(a)$ and $t\mapsto\beta_t(b)$ are continuous for every $a$ in $A$ and every $b$ in $B$, and the diagram
    \begin{center}
        \begin{tikzcd}
A \arrow[r, "\alpha_t"] \arrow[d, "\phi"] &   C \arrow[d, "\psi"]              \\
B \arrow[r, "\beta_t"]              & D           
\end{tikzcd}
    \end{center}
    is commutative for every $0\leq t\leq1$. Then $(\alpha_0,\beta_0)_*=(\alpha_1,\beta_1)_*$.
    \item Stability: if $p$ is any rank one projection in $\K$ and $\kappa_A:A\to A\otimes\K$ is defined by $\kappa_A(a)=a\otimes p$ for every C$^*$-algebra $A$, then the morphism
    \begin{center}
\begin{tikzcd}
A \arrow[r, "\kappa_A"] \arrow[d, "\phi"] & A\otimes\K \arrow[d, "\phi\otimes\id_\K"] \\
B \arrow[r, "\kappa_B"]                   & B\otimes\K                               
\end{tikzcd}
    \end{center}
    induces an isomorphism \emph{$(\kappa_A,\kappa_B)_*:K_*(\phi)\to K_*(\phi\otimes\id_\K)$}.
    \item Continuity: suppose that $(A,\mu_i)$ is the inductive limit (in the category of C$^*$-algebras and $^*$-homomorphisms) of the inductive system $(A_i,\alpha_{ij})$, and $(B,\nu_i)$ is likewise the inductive limit of the inductive system $(B_i,\beta_{ij})$. Suppose also that, for each pair of indices $i\leq j$, there are $^*$-homomorphisms $\phi_i:A_i\to B_i$ and $\phi_j:A_j\to B_j$ such that the diagram
    \begin{center}
\begin{tikzcd}
A_i \arrow[r, "\alpha_{ij}"] \arrow[d, "\phi_i"] & A_j \arrow[d, "\phi_j"] \\
B_i \arrow[r, "\beta_{ij}"]                      & B_j                        
\end{tikzcd}
    \end{center}
    is commutative. Then there exists a $^*$-homomorphism $\phi:A\to B$ such that $(\phi,(\mu_i,\nu_i))$ is the inductive limit of $(\phi_i,(\alpha_{ij},\beta_{ij}))$ in the category \textbf{C$^*$-hom}, and $(K_*(\phi),(\mu_i,\nu_i)_*)$ is isomorphic to the inductive limit of $(K_*(\phi_i),(\alpha_{ij},\beta_{ij})_*)$ in the category of abelian groups.
    \end{enumerate}
    \end{theorem}
    
    We also collect some excision results that follow easily from Theorem \ref{main}. When $A$ is a C$^*$-subalgebra of $B$ and $\phi:A\to B$ is the inclusion map, we denote $K_*(\phi)$ by $K_*(A,B)$.    
    \begin{theorem}\label{exc} Let $\phi:A\to B$ be a $^*$-homomorphism of C$^*$-algebras $A$ and $B$.
    \begin{enumerate}[(i)]
	\item If $K_*(B)=0$, the maps $\nu_0:K_0(\phi)\to K_0(A)$ and $\nu_1:K_1(\phi)\to K_1(A)$ in part (i) of Theorem \ref{main} are isomorphisms.
        \item If $K_*(A)=0$, the maps $\mu_0:K_1(B)\to K_0(\phi)$ and $\mu_1:K_0(A)\to K_1(\phi)$ in part (i) of Theorem \ref{main} are isomorphisms.
	\item If $\phi$ is surjective, the morphism
	\begin{center}
\begin{tikzcd}
\ker\phi \arrow[d] \arrow[r, "\iota_\phi"] & A \arrow[d, "\phi"] \\
0 \arrow[r]                & B          
\end{tikzcd}
	\end{center}
	where $\iota_\phi:\ker\phi\to A$ is the inclusion map, induces a natural isomorphism $(\iota_\phi,0)_*:K_*(\ker\phi)\to K_*(\phi)$.
	\item If $I$ is a closed, two-sided ideal in the C$^*$-algebra $A$, then there are natural isomorphisms $K_0(A/I)\cong K_1(I,A)$ and $K_1(A/I)\cong K_0(I,A)$.
\end{enumerate}
\end{theorem}
	
	We now discuss some examples to illustrate the utility of parts (i) and (iii) of Theorem \ref{main}.

\begin{example}\label{diag}
Let $D$ be any C$^*$-algebra, and let $A$ be the subalgebra of $B=M_2(D)$ consisting of the diagonal matrices. Since $K_*(A)\cong K_*(D)\oplus K_*(D)$ and $K_*(B)\cong K_*(D)$, we may write the six-term exact sequence of part (i) of Theorem \ref{main} as

\begin{center}
\begin{tikzcd}
K_1(D) \arrow[rr] &  & K_0(A,B) \arrow[rr] &  & K_0(D)\oplus K_0(D) \arrow[dd, "\phi_*"] \\
             &  &                      &  &                                   \\
K_1(D)\oplus K_1(D) \arrow[uu, "\phi_*"] &  & K_1(A,B) \arrow[ll] &  & K_0(D) \arrow[ll]                   
\end{tikzcd}
\end{center}

The vertical maps are both $\phi_*(g,h)=g+h$. Exactness implies that $K_*(A,B)\cong\ker\phi_*\cong K_*(D)$.

As a special case of interest, let $\H$ be a separable Hilbert space of dimension at least $2$, and $\M$ a closed subspace such that $\M\neq\{0\}$ and $\M\neq\H$. Let $A=\K(\M)\oplus\K(\M^\perp)$, where $\K(\M)$ is the C$^*$-algebra of compact operators on $\M$, regarded as a subalgebra of $B=\K(\H)$ as operators that leave $\M$ and $\M^\perp$ invariant. Then $K_0(A,B)\cong\bZ$ and $K_1(A,B)=0$. If we fix a unit vector $\xi$ in $\M$, a unit vector $\eta$ in $\M^\perp$, and a partial isometry $v$ in $B$ with source subspace $\spn\{\xi\}$ and range subspace $\spn\{\eta\}$, the group $K_0(A,B)$ is generated by the class of the triple $(v^*v,vv^*,v)$.
\end{example}

\begin{example}\label{gamma}
Let $D$ be any C$^*$-algebra and consider $A=D$ as a subalgebra of $B=D\oplus D$ via the embedding $d\mapsto(d,d)$. The six-term exact sequence of part (i) of Theorem \ref{main} becomes

\begin{center}
\begin{tikzcd}
K_1(D)\oplus K_1(D) \arrow[rr] &  & K_0(A,B) \arrow[rr] &  & K_0(D) \arrow[dd, "\phi_*"] \\
             &  &                      &  &                                   \\
K_1(D) \arrow[uu, "\phi_*"] &  & K_1(A,B) \arrow[ll] &  & K_0(D)\oplus K_0(D) \arrow[ll]                   
\end{tikzcd}
\end{center}

This time the vertical maps are $\phi_*(g)=(g,g)$, which are injective, whence exactness implies $K_0(A,B)\cong K_1(D)$ and $K_1(A,B)\cong K_0(D)$. In the case that $D=\K$, the group $K_1(A,B)\cong\bZ$ is generated by the class of the triple $(p,p,g)$, where $p$ is a rank one projection in $\K$ and $g(s)=(e^{2\pi is}p,p)$. Observe that we do not need to consider the unit in the unitization $\tilde\K$ to describe the group $K_1(A,B)$.

\end{example}

\begin{example}
Consider the diagram
\begin{center}
\begin{tikzcd}
0 \arrow[r] & C_0(\bR^2) \arrow[r] \arrow[d, "\psi"] & C(\bD) \arrow[r] \arrow[d, "\phi"] & C(\bT) \arrow[r] \arrow[d, hook] & 0 \\
0 \arrow[r] & 0 \arrow[r]                    & {C([0,1])} \arrow[r, equal]                 & {C([0,1])} \arrow[r]         & 0
\end{tikzcd}
\end{center}
where $C_0(\bR^2)$ is identified with functions that vanish on the boundary of $\bD=\{(x,y)\in\bR^2\mid x^2+y^2\leq1\}$. The algebra $C(\bT)$ is viewed as all $f$ in $C([0,1])$ with $f(0)=f(1)$, and $\phi$ is the composition of the restriction to the boundary $C(\bD)\to C(\bT)$ with the inclusion $C(\bT)\hookrightarrow C([0,1])$. We have $K_0(\phi)=K_1(\phi)=0$ since $K_1(C([0,1]))=K_1(C(\bD))=0$ (see Corollary \ref{zero}) and the induced map $\phi_*:K_0(C(\bD))\to K_0(C([0,1]))$ is an isomorphism. The six-term exact sequence of part (iii) of Theorem \ref{main} becomes
\begin{center}
\begin{tikzcd}
K_0(C_0(\bR^2)) \arrow[rr] &  & 0 \arrow[rr] &  & K_0(C(\bT),C([0,1])) \arrow[dd, "\partial_0"]  \\
                       &  &                      &  &                        \\
K_1(C(\bT),C([0,1])) \arrow[uu, "\partial_1"]  &  & 0 \arrow[ll] &  & K_1(C_0(\bR^2)) \arrow[ll]
\end{tikzcd}
	\end{center}
(we identify $K_*(\psi)$ with $K_*(C_0(\bR^2))$ using Theorem \ref{exc}). It can be shown that $K_1(C(\bT),C([0,1]))\cong\bZ$ is generated by the class of $(1,z,g)$ where $z$ is the function $z\mapsto z$ on $\bT$ and $g(t)=f_t$, where $f_t(s)=e^{2\pi ist}$. Using the notation in Definition \ref{index}, let $l=1$,
\[w=\left[\begin{array}{cc}
    z & -(1-|z|^2)^{1/2} \\
    (1-|z|^2)^{1/2} & \ol z
\end{array}\right]\]
and $h=g$. Then
\[\partial_1([1,z,g])=\left[\left[\begin{array}{cc}
    |z|^2 & z(1-|z|^2)^{1/2} \\
    \ol z(1-|z|^2)^{1/2} & 1-|z|^2
\end{array}\right]\right]-\left[\left[\begin{array}{cc}
    1 & 0 \\
    0 & 0
\end{array}\right]\right]\]
\end{example}

\begin{example}
Let $D$ be the diagonal matrices in $M_2(\bC)$. Consider the diagram
\begin{center}
\begin{tikzcd}
0 \arrow[r] & C_0(\bR) \arrow[r] \arrow[d] & C([0,1]) \arrow[r, "\pi"] \arrow[d, "\phi"] & D \arrow[r] \arrow[d, hook] & 0 \\
0 \arrow[r] & 0 \arrow[r]                    & M_2(\bC) \arrow[r, equal]                 & M_2(\bC) \arrow[r]         & 0
\end{tikzcd}
\end{center}
Where $C_0(\bR)$ is identified with all functions in $C([0,1])$ that vanish at the endpoints, the map $\pi:C([0,1])\to D$ is defined by
\[\pi(f)=\left[\begin{array}{cc}
    f(0) & 0 \\
    0 & f(1)
\end{array}\right],\]
and $\phi$ is the composition of $\pi$ with the inclusion $D\hookrightarrow M_2(\bC)$. We have $K_0(C_0(\bR))=0$ and $K_1(D,M_2(\bC))=0$, the latter by Example \ref{diag}. The six-term exact sequence of part (iii) of Theorem \ref{main} becomes
\begin{center}
\begin{tikzcd}
0 \arrow[rr] &  & K_0(\phi) \arrow[rr] &  & K_0(D,M_2(\bC)) \arrow[dd, "\partial_0"]  \\
                      &  &                      &  &                        \\
0 \arrow[uu]  &  & K_1(\phi) \arrow[ll] &  & K_1(C_0(\bR)) \arrow[ll]
\end{tikzcd}
	\end{center}
Using part (i) of Theorem \ref{main}, it can be shown that $K_0(\phi)=0$ and $K_1(\phi)\cong\bZ/2\bZ$, with the nontrivial element in $K_1(\phi)$ given by the class of the triple $(1,1,g)$, where
\[g(s)=\left[\begin{array}{cc}
    e^{2\pi is} & 0 \\
    0 & 1
\end{array}\right]\]
The map $\partial_0$ is therefore injective and takes a generator of $K_0(D,M_2(\bC))\cong\bZ$ to twice a generator of $K_1(C_0(\bR))\cong\bZ$. More concretely,
\[\partial_0\left(\left[\left[\begin{array}{cc}
    0 & 0 \\
    0 & 1
\end{array}\right],\left[\begin{array}{cc}
    1 & 0 \\
    0 & 0
\end{array}\right],\left[\begin{array}{cc}
    0 & 1 \\
    0 & 0
\end{array}\right]\right]\right)=-[z^2]\]
where $z$ is the function $z\mapsto z$ on $C(\bT)\cong C_0(\bR)^\sim$.
\end{example}

The final example illustrates the excision theorem of \cite{putnam3}. We refer to \cite{groupoid} for a detailed account of constructing a C$^*$-algebra $C_r^*(R)$ from an equivalence relation $R$, or, more generally, from an \'etale groupoid.

\begin{example}
Let $X=\{0,1\}^\bN$, the space of all sequences of $0$'s and $1$'s with the product topology. Define the surjective map $\omega:X\to\bT$ by
\[\omega(\{x_n\})=\exp\left(2\pi i\sum_{n=1}^\infty x_n2^{-n}\right).\]
Define $S\subseteq X\times X$ to be \emph{tail-equivalence} on $X$, that is,
\[(\{x_n\},\{y_n\})\in S\Longleftrightarrow x_n=y_n\text{ for sufficiently large }n.\]
The equivalence relation $S$ has a natural topology under which it is a Hausdorff, \'etale groupoid, and $C_r^*(S)$ is isomorphic to the UHF algebra $M_{2^\infty}$. Define $T=\omega\times\omega(S)$; it is a consequence of the main result of \cite{haslehurst} that $T$, with the quotient topology from $S$ and the map $\omega\times\omega|_S$, is a Hausdorff, \'etale groupoid. It has the following concrete description, as can be easily checked:
\[\textstyle T=\{(w,z)\in\bT\times\bT\mid w=e^{2\pi i\theta}z\text{ for some }\theta\in\bZ[\frac12]\},\]
where $\bZ\left[\frac12\right]=\left\{\frac{k}{2^n}\mid k\in\bZ\AND n\geq0\right\}$. Moreoever, the map $\omega\times\omega$ satisfies the standing hypotheses of section 7 of \cite{putnam3}, and thus induces an injective $^*$-homomorphism from $C_r^*(T)$ to $C_r^*(S)$, enabling us to regard $C_r^*(T)$ as a C$^*$-subalgebra of $C_r^*(S)$.

Let $\bT_D=\{e^{2\pi i\theta}\mid\theta\in\bZ[\frac12]\}$. Notice that $\bT_D$ is exactly where the map $\omega$ is not one-to-one. The main result from section 7 of \cite{putnam3} says (avoiding some technical details) that the relative $K$-theory of the inclusion $C_r^*(T)\subseteq C_r^*(S)$ can be computed by looking only where the map $\omega\times\omega$ is not one-to-one. More precisely,
\[K_*(C_r^*(T),C_r^*(S))\cong K_*(\K(l^2(\bT_D)),\K(l^2(\bT_D))\oplus\K(l^2(\bT_D))).\]
Where we regard $\K(l^2(\bT_D))$ as a C$^*$-subalgebra of $\K(l^2(\bT_D))\oplus\K(l^2(\bT_D))$ as in Example \ref{gamma}. The six-term exact sequence of part (i) of Theorem \ref{main} becomes
\begin{center}
\begin{tikzcd}
0 \arrow[rr] &  & K_0(C_r^*(T),C_r^*(S)) \arrow[rr] &  & K_0(C_r^*(T)) \arrow[dd]  \\
                      &  &                      &  &                        \\
K_1(C_r^*(T)) \arrow[uu]  &  & K_1(C_r^*(T),C_r^*(S)) \arrow[ll] &  & K_0(C_r^*(S)) \arrow[ll]
\end{tikzcd}
	\end{center}
Example \ref{gamma} allows us to conclude that $K_0(C_r^*(T),C_r^*(S))=0$ and $K_1(C_r^*(T),C_r^*(S))\cong\bZ$. It may also be shown that the right vertical map is surjective, see Lemma 5.5 of \cite{haslehurst}. We thus obtain that $K_0(C_r^*(T))\cong\bZ[\frac12]$ (with the order inherited from $\bR$) and $K_1(C_r^*(T))\cong\bZ$.
\end{example}

The previous example illustrates the simplest case of a general construction involving Bratteli diagrams; we refer to section 3 of \cite{haslehurst} for more information.

\section{Definitions and a portrait of $K_*(\phi)$}

We begin by establishing some notation and terminology. If $A$ is a C$^*$-algebra, we let $\tilde A$ denote its unitization. If $a$ is in $\tilde A$, let $\dot a$ denote the scalar part of $a$. Let $M_n(A)$ denote the $n\times n$ matrices with entries in $A$, regarded as a C$^*$-algebra in the usual way. Let $M_\infty(\tilde A)$ be the union $\bigcup_{n=1}^\infty M_n(\tilde A)$, which may be regarded as an increasing union by means of the inclusions $M_n(\tilde A)\subseteq M_{n+1}(\tilde A)$, $a\mapsto\diag(a,0)$. If $a$ and $b$ are in $M_\infty(\tilde A)$ we define
	\[a\oplus b=\left[\begin{array}{cc}
	    a & 0 \\
	    0 & b
	\end{array}\right]\]
	Admittedly, there is some ambiguity in the above definition of $a\oplus b$ since $a$ and $b$ may be regarded as matrices of arbitrarily large size. However, since $K$-theory doesn't distinguish elements that are ``moved down the diagonal", there will be negligible harm done by ignoring this technical issue. We denote by $1_n$ the identity matrix in $M_n(\tilde A)$, or the matrix in $M_\infty(\tilde A)$ with $n$ consecutive occurences of $1$ down the diagonal, and $0$ elsewhere.
	
	An element $p$ of a C$^*$-algebra is called a \textit{projection} if $p=p^2=p^*$. An element $u$ of a unital C$^*$-algebra is called a \textit{unitary} if $u^*u=uu^*=1$. An element $v$ of a C$^*$-algebra is called a \textit{partial isometry} if $v^*v$ is a projection (in which case, $vv^*$ is also a projection). Two projections $p$ and $q$ in $M_\infty(\tilde A)$ are called \emph{Murray-von Neumann equivalent} if there is a partial isometry $v$ in $M_\infty(\tilde A)$ such that $v^*v=p$ and $vv^*=q$, and we will say in this situation that $v$ \textit{is a partial isometry from $p$ to $q$}. It is straightforward to check the useful formulas $v=qv=vp=qvp$.
	
	The group $K_0(\tilde A)$ is the Grothendieck completion of the semigroup of Murray-von Neumann classes of projections in $M_\infty(\tilde A)$ with the operation $[p]+[q]=[p\oplus q]$. The group $K_0(A)$ is the kernel of the map $K_0(\tilde A)\to\bZ$ induced by the scalar map $\tilde A\to\bC$. The group $K_1(A)$ is the group of stable homotopy classes of unitaries over $\tilde A$ with the operation $[u]+[v]=[uv]$ (regard $u$ and $v$ as elements of the same matrix algebra so that the product is well-defined). Every element of $K_0(A)$ may be represented by a formal difference $[p]-[q]$ of classes such that $\dot p=1_n$ and $q=1_n$ for some $n$, see the discussion following 5.5.1 in \cite{blackadar}. Every element of $K_1(A)$ may be represented by a class $[u]$ such that $u$ is in $M_n(\tilde A)$ and $\dot u=1_n$.
	
	We denote the \textit{compact operators} on a separable Hilbert space by $\K$. We write $SA$ for $C_0((0,1))\otimes A=C_0((0,1),A)$, the \emph{suspension} of $A$, and $CA$ for $C_0((0,1])\otimes A=C_0((0,1],A)$, the \textit{cone} of $A$. If $\phi:A\to B$ is a $^*$-homomorphism, we commit the usual notation abuse and denote the obvious induced maps $\tilde A\to\tilde B$, $SA\to SB$, $CA\to CB$, $M_n(A)\to M_n(B)$ (or any combination of these) by $\phi$. Clarity will sometimes be needed for the first three, in which case they will be denoted by $\tilde\phi$, $S\phi$, and $C\phi$ respectively. The \emph{mapping cone} $C_\phi$ of $\phi$ is defined to be the pullback of $\phi$ and the  map $\pi_B:CB\to B$, $\pi_B(f)=f(1)$. In other words, it is the C$^*$-algebra
    \[C_\phi=\{(a,f)\mid f(1)=\phi(a)\}\subseteq A\oplus CB.\]
    We denote the induced maps $K_j(A)\to K_j(B)$ by $\phi_*$ for both $j=0,1$. We denote the natural isomorphism $K_1(A)\to K_0(SA)$ by $\theta_A$ and the Bott map $K_0(A)\to K_1(SA)$ by $\beta_A$. If
	\begin{center}
\begin{tikzcd}
0 \arrow[r] & I \arrow[r, "\iota"] & A \arrow[r, "\pi"] & A/I \arrow[r] & 0
\end{tikzcd}
	\end{center}
	is a short exact sequence of C$^*$-algebras, we denote the index maps $K_n(A/I)\to K_{n-1}(I)$ by $\delta_n$ for $n\geq1$ and the exponential map $K_0(A/I)\to K_1(I)$ by $\delta_0$. We refer the reader to \cite{blackadar} or \cite{rordam} for more details on $K$-theory for C$^*$-algebras.
	
	\begin{definition}\label{big}
Define $\Gamma_0(\phi)$ to be the set of all triples $(p,q,v)$ where $p$ and $q$ are projections in $M_\infty(\tilde A)$ and $v$ is a partial isometry in $M_\infty(\tilde B)$ from $\phi(p)$ to $\phi(q)$ (recall this means that $v^*v=\phi(p)$ and $vv^*=\phi(q)$). For brevity, we will often denote these triples by the symbols $\sigma$ and $\tau$.
\begin{enumerate}[(i)]
    \item Define the direct sum operation $\oplus$ on $\Gamma_0(\phi)$ by
	\[(p,q,v)\oplus(p',q',v')=(p\oplus p',q\oplus q',v\oplus v').\]
	\item We say that two triples $(p,q,v)$ and $(p',q',v')$ are \emph{isomorphic}, written $(p,q,v)\cong(p',q',v')$, if there exist partial isometries $c$ and $d$ from $p$ to $p'$ and from $q$ to $q'$, respectively, that intertwine $v$ and $v'$, that is, $\phi(d)v=v'\phi(c)$.
 \item A triple $(p,q,v)$ is called \emph{elementary} if $p=q$ and there is a homotopy $v_t$ for $0\leq t\leq1$ such that $v_0=\phi(p)$, $v_1=v$, and $v_t^*v_t=v_tv_t^*=\phi(p)$ for all $t$.
	\item Two triples $\sigma$ and $\sigma'$ in $\Gamma_0(\phi)$ are \emph{equivalent}, written $\sigma\sim\sigma'$, if there exist elementary triples $\tau$ and $\tau'$ such that $\sigma\oplus\tau\cong\sigma'\oplus\tau'$.
\end{enumerate}
	   Denote by $[\sigma]$, or $[p,q,v]$, the equivalence class of the triple $\sigma=(p,q,v)$ via the relation $\sim$. $K_0(\phi)$ is then defined to be the quotient of $\Gamma_0(\phi)$ by the relation $\sim$, that is,
	\[\{[\sigma]\mid\sigma\in\Gamma_0(\phi)\}=\Gamma_0(\phi)/\sim\]
	\end{definition}
	
    We make two simple observations. First, the notions of isomorphism and elementary for triples behave well with respect to the direct sum operation: if $\sigma_1\cong\sigma_2$ and $\sigma_3\cong\sigma_4$, then $\sigma_1\oplus\sigma_3\cong\sigma_2\oplus\sigma_4$, for any two triples $\sigma$ and $\sigma'$, we have $\sigma\oplus\sigma'\cong\sigma'\oplus\sigma$, and if $\sigma$ and $\sigma'$ are elementary, then so is $\sigma\oplus\sigma'$. Second, all elementary triples are equivalent to each other, and two isomorphic triples are equivalent.
    
    We recall the following useful fact: if $u$ is a self-adjoint unitary in a unital C$^*$-algebra $A$, then $u=e^{i\pi(1-u)/2}$. To see this, note that $(1-u)^2=2(1-u)$, so $(1-u)^n=2^{n-1}(1-u)$ by induction, and hence
    \[e^{i\pi(1-u)/2}=\sum_{n=0}^\infty\frac{(i\pi(1-u)/2)^n}{n!}=1+\sum_{n=1}^\infty\frac12\frac{(i\pi)^n}{n!}(1-u)=1+\frac12(e^{i\pi}-1)(1-u)=u\]
    it follows that $u$ is homotopic to $1$ via the path $e^{i\pi t(1-u)/2}$ for $0\leq t\leq1$.
	   
	\begin{prop}\label{group} $K_0(\phi)$ is an abelian group when equipped with the binary operation
	\[[\sigma]+[\sigma']=[\sigma\oplus\sigma']\]
	where the identity element is given by $[0,0,0]$, and the inverse of $[p,q,v]$ is given by $[q,p,v^*]$.
	\end{prop}
	
	\begin{proof} That $K_0(\phi)$ is an abelian group follows quite readily from the observations above, and the fact that $[0,0,0]$ is the identity element is all but trivial since we identify $a$ with $a\oplus0$ in $M_\infty(\tilde A)$. To prove the last statement, note that
	\[[p,q,v]+[q,p,v^*]=[p\oplus q,q\oplus p,v\oplus v^*]\]
	and the triple $(p\oplus q,q\oplus p,v\oplus v^*)$ is isomorphic to the triple
	\begin{equation}\label{slfadjunt}
	\left(\left[\begin{array}{cc}
	   p & 0 \\
	    0 & q
	\end{array}\right],\left[\begin{array}{cc}
	    p & 0 \\
	    0 & q
	\end{array}\right],\left[\begin{array}{cc}
	    0 & v^* \\
	    v & 0
	\end{array}\right]\right)
	\end{equation}
	by taking
	\[d=\left[\begin{array}{cc}
	   0 & p \\
	    q & 0
	\end{array}\right],\qquad c=\left[\begin{array}{cc}
	   p & 0 \\
	    0 & q
	\end{array}\right]\]
	in part (ii) of Definition \ref{big}. The partial isometry in (\ref{slfadjunt}) is a self-adjoint unitary in the C$^*$-algebra $(\phi(p)\oplus\phi(q))M_\infty(\tilde B)(\phi(p)\oplus\phi(q))$, and is thus homotopic to the identity $\phi(p)\oplus\phi(q)$. Therefore, (\ref{slfadjunt}) is elementary, its class is zero, and hence the class of $(p\oplus q,q\oplus p,v\oplus v^*)$ is zero because it is isomorphic to (\ref{slfadjunt}).
	\end{proof}
	
	We collect some useful properties of the elements of $K_0(\phi)$.
	   
	   \begin{prop}\label{3}
	   \begin{enumerate}[(i)]
	       \item Suppose that $p$ and $q$ are projections in $M_\infty(\tilde A)$ and $v$ and $v'$ are two partial isometries in $M_\infty(\tilde B)$ from $\phi(p)$ to $\phi(q)$. If there is a continuous path $v_t$ of partial isometries such that $v_0=v$, $v_1=v'$, and $v_t^*v_t=\phi(p)$ and $v_tv_t^*=\phi(q)$ for all $0\leq t\leq1$, then $[p,q,v]=[p,q,v']$.
	       \item Suppose that $p$, $q$, and $r$ are projections in $M_\infty(\tilde A)$ and $v$ and $w$ are partial isometries in $M_\infty(\tilde B)$ from $\phi(p)$ to $\phi(q)$ and from $\phi(q)$ to $\phi(r)$, respectively. Then
	       \[[p,q,v]+[q,r,w]=[p,r,wv].\]
	   \item Let $(p,q,v)$ and $(p',q',v')$ be two triples in $\Gamma_0(\phi)$. If $pp'=0$, then
	\[(p,q,v)\oplus(p',q',v')\cong\left(p+p',q\oplus q',\left[\begin{array}{cc}
	     v & 0\\
	     v' & 0
	\end{array}\right]\right)\]
	If $qq'=0$, then
	\[(p,q,v)\oplus(p',q',v')\cong\left(p\oplus p',q+q',\left[\begin{array}{cc}
	     v & v' \\
	     0 & 0
	\end{array}\right]\right)\]
	If $pp'=qq'=0$, then
	\[(p,q,v)\oplus(p',q',v')\cong(p+p',q+q',v+v')\]
	   \item Every triple in $\Gamma_0(\phi)$ is equivalent to one of the form $(p,1_n,v)$, where $n\geq1$ and $\dot p=\dot v=1_n$, and one of the form $(1_n,q,v)$, where $n\geq1$ and $\dot q=\dot v=1_n$.
	   \item $[p,q,v]=0$ if and only if there exist projections $r$ and $s$ in $M_\infty(\tilde A)$ and partial isometries $x$ and $y$ in $M_\infty(\tilde A)$ from $p\oplus r$ to $s$ and $q\oplus r$ to $s$, respectively, such that $\phi(y)(v\oplus\phi(r))\phi(x^*)$ and $\phi(s)$ are homotopic as unitaries through $\phi(s)M_\infty(\tilde B)\phi(s)$.
	   
	   Suppose $m\geq n$, $p$ is in $M_m(\tilde A)$, and $(p,1_n,v)$ is a triple in $\Gamma_0(\phi)$ with $\dot p=\dot v=1_n$. Then $[p,1_n,v]=0$ if and only if there exist $k\geq0$ and a partial isometry $w$ in $M_{m+k}(\tilde A)$ with $w^*w=\dot w=1_n\oplus0_{m-n}\oplus1_k$ and $ww^*=p\oplus1_k$ such that $(v\oplus1_k)\phi(w)$ and $1_n\oplus0_{m-n}\oplus1_k$ are homotopic as unitaries in $(1_n\oplus0_{m-n}\oplus1_k)M_{m+k}(\tilde B)(1_n\oplus0_{m-n}\oplus1_k)$.
	\end{enumerate}
	   \end{prop}
	   
	   \begin{proof} \begin{enumerate}[(i)]
	      \item We have
	   \[[p,q,v]-[p,q,v']=[p,q,v]+[q,p,v'^*]=[p\oplus q,q\oplus p,v\oplus v'^*]\]
	  and the triple $(p\oplus q,q\oplus p,v\oplus v'^*)$ is isomorphic to
	  \[\left(\left[\begin{array}{cc}
	   p & 0 \\
	   0 & q
	\end{array}\right],\left[\begin{array}{cc}
	    p & 0 \\
	    0 & q
	\end{array}\right],\left[\begin{array}{cc}
	    0 & v'^* \\
	    v & 0
	\end{array}\right]\right)\]
	similarly as in the proof of Proposition \ref{group}. The partial isometry in the above triple is homotopic to a self-adjoint unitary in $(\phi(p)\oplus\phi(q))M_\infty(\tilde B)(\phi(p)\oplus\phi(q))$ via the path
	\[\left[\begin{array}{cc}
	    0 & v_t^* \\
	    v & 0
	\end{array}\right].\]
	It follows that the triple is elementary.
	\item We have
	   \[[p,q,v]+[q,r,w]=[p\oplus q,q\oplus r,v\oplus w]\]
	   and observe that the triple $(p\oplus q,q\oplus r,v\oplus w)$ is isomorphic to the triple
	   \[\left(\left[\begin{array}{cc}
	   p & 0 \\
	    0 & q
	\end{array}\right],\left[\begin{array}{cc}
	    r & 0 \\
	    0 & q
	\end{array}\right],\left[\begin{array}{cc}
	    0 & w \\
	    v & 0
	\end{array}\right]\right).\]
    by taking
	\[d=\left[\begin{array}{cc}
	   0 & r \\
	    q & 0
	\end{array}\right],\qquad c=\left[\begin{array}{cc}
	   p & 0 \\
	    0 & q
	\end{array}\right]\]
	in part (ii) of Definition \ref{big}. We also have that $[p,r,wv]=[p\oplus q,r\oplus q,wv\oplus\phi(q)]$ since $(q,q,\phi(q))$ is elementary. Now
	\[\left[\begin{array}{cc}
	    0 & w \\
	    v & 0
	\end{array}\right]=\left[\begin{array}{cc}
	    0 & w \\
	    w^* & 0
	\end{array}\right]\left[\begin{array}{cc}
	    wv & 0 \\
	    0 & \phi(q)
	\end{array}\right]\]
	and the second matrix above is homotopic to $\phi(r)\oplus\phi(q)$. It follows that the two matrices
	\[\left[\begin{array}{cc}
	    0 & w \\
	    v & 0
	\end{array}\right]\qquad\left[\begin{array}{cc}
	    wv & 0 \\
	    0 & \phi(q)
	\end{array}\right]\]
	are homotopic, and hence, by part (i), the two triples
	\[\left(\left[\begin{array}{cc}
	   p & 0 \\
	    0 & q
	\end{array}\right],\left[\begin{array}{cc}
	    r & 0 \\
	    0 & q
	\end{array}\right],\left[\begin{array}{cc}
	    0 & w \\
	    v & 0
	\end{array}\right]\right),\qquad\left(\left[\begin{array}{cc}
	   p & 0 \\
	    0 & q
	\end{array}\right],\left[\begin{array}{cc}
	    r & 0 \\
	    0 & q
	\end{array}\right],\left[\begin{array}{cc}
	    wv & 0 \\
	    v & \phi(q)
	\end{array}\right]\right)\]
	are equivalent.
	\item If $pp'=0$, then $v\phi(p')=v\phi(p)\phi(p')=0$, and similarly $v'\phi(p)=0$. Thus we have \[\left[\begin{array}{cc}
	    \phi(q) & 0 \\
	    0 & \phi(q')
	\end{array}\right]\left[\begin{array}{cc}
	    v & 0 \\
	    0 & v'
	\end{array}\right]=\left[\begin{array}{cc}
	    v & 0 \\
	    v' & 0
	\end{array}\right]\left[\begin{array}{cc}
	    \phi(p) & \phi(p') \\
	    0 & 0
	\end{array}\right]\]
	so the triples are isomorphic. The other two claims are similar.
	\item
	We show only that any triple is equivalent to a triple of the second form $(1_n,q,v)$, where $n\geq1$ and $\dot q=\dot v=1_n$, since the proof is analogous for the first form. Take any triple $(p,q,v)$, and choose $n\geq1$ so that $p$ is in $M_n(\tilde A)$. By adding the elementary triple $(1_n-p,1_n-p,1_n-\phi(p))$ and using part (iii),
	\[(p,q,v)\sim\left(\left[\begin{array}{cc}
	     1_n & 0 \\
	     0 & 0
	\end{array}\right],\left[\begin{array}{cc}
	     q & 0 \\
	     0 & 1_n-p
	\end{array}\right],\left[\begin{array}{cc}
	     v & 0 \\
	     1_n-\phi(p) & 0
	\end{array}\right]\right).\]
	Set $q_1=q\oplus(1_n-p)$ and $v_1=\left[\begin{array}{cc}
	     v & 0 \\
	     1_n-\phi(p) & 0
	\end{array}\right]$. We have $\dot v_1^*\dot v_1=1_n$ and $\dot v_1\dot v_1^*=\dot q_1$, so choose $m\geq n$ and a unitary $u$ in $M_m(\bC)$ such that $u\dot q_1u^*=1_n$. Then $(1_n,q_1,v_1)\cong(1_n,uq_1u^*,uv_1)$ since $uv_1\phi(1_n)=\phi(uq_1)v_1$. The scalar part of $uq_1u^*$ is $1_n$ and we now set $q_2=uq_1u^*$ and $v_2=uv_1$. Lastly, $\dot v_2$ may then be regarded as a unitary in $M_n(\bC)$, so choose a homotopy $v_t$ from $\dot v_2$ to $1_n$ and observe that each $v_2v_t^*$ is a partial isometry from $1_n$ to $\phi(q_2)$. By part (i), $(1_n,q_2,v_2)\sim(1_n,q_2,v_2\dot v_2^*)$. Setting $v_3=v_2\dot v_2^*$, the triple $(1_n,q_2,v_3)$ has the desired properties.
	   \item For the first part, it is a direct consequence of the definitions that $[p,q,v]=0$ if and only if there are elementary triples $(r,r,c)$ and $(s,s,d)$ such that
	\[(p,q,v)\oplus(r,r,c)\cong(s,s,d)\]
	This is true if and only if there are partial isometries $x$ and $y$ in $M_\infty(\tilde A)$ such $d\phi(x)=\phi(y)(v\oplus c)$. Then $d=\phi(y)(v\oplus c)\phi(x^*)$, and since $d$ is homotopic to $\phi(s)$ and $c$ is homotopic to $\phi(r)$, we have the conclusion.
	
	For the second part, for appropriate $m\geq n$, obtain elementary triples $(r,r,c)$ and $(s,s,d)$ such that
	\[(p,1_n\oplus0_{m-n},v)\oplus(r,r,c)\cong(s,s,d).\]
	By replacing $s$ with $s\oplus(1_k-r)$ and $d$ by $d\oplus(1_k-\phi(r))$, and using part (iii), we may assume that $r=1_k$ for some $k\geq0$. Obtain $x$ and $y$ as in the previous paragraph, so that $d\phi(x)=\phi(y)(v\oplus c)$, hence $\phi(y^*)d\phi(y)=(v\oplus c)\phi(x^*y)$. Since $d$ is homotopic to $\phi(s)$, $\phi(y^*)d\phi(y)$ is homotopic to $1_n\oplus0_{m-n}\oplus1_k$. Also, $c$ is homotopic to $1_k$. Therefore, $w=x^*y\dot y^*\dot x$ has the desired properties.\qedhere
	   \end{enumerate}
	   \end{proof}
	   
	   In II.3.3 of \cite{karoubi}, Karoubi introduces a definition of the $K_1$-group, there denoted $K^{-1}(\C)$ for a Banach category $\C$, that gives an equivalent but slightly more general description. We provide the definition in order to motivate the definition of the relative $K_1$-group. Consider the set $\Gamma_1(A)$ of all pairs $(p,u)$ such that $p$ is a projection in $M_\infty(\tilde A)$ and $u$ is a unitary in $pM_\infty(\tilde A)p$. Define the direct sum $(p,u)\oplus(p',u')=(p\oplus p',u\oplus u')$, as usual. Say that two pairs $(p,u)$ and $(p',u')$ are isomorphic, written $(p,u)\cong(p',u')$, if there is a partial isometry $v$ from $p$ to $p'$ such that $vu=u'v$. We say a pair $(p,u)$ is elementary if there is a continuous path of unitaries $u_t$ from $u$ to $p$ through $pM_\infty(\tilde A)p$. We say that two pairs $\sigma$ and $\sigma'$ in $\Gamma_1(A)$ are equivalent, written $\sigma\sim\sigma'$, if there exist elementary pairs $\tau$ and $\tau'$ such that $\sigma\oplus\tau\cong\sigma\oplus\tau'$. Denote by $[\sigma]$, or $[p,u]$, the equivalence class of the pair $\sigma=(p,u)$ via the relation $\sim$. $K^{-1}(\C_A)$ is defined to be the quotient of $\Gamma_1(A)$ by the relation $\sim$. It is an abelian group with $[0,0]=0$ and $-[p,u]=[p,u^*]$.
	   
	   The proof of the following result uses similar, but simpler, techniques to those in Proposition \ref{3}. For this reason, and because we will not need it, we omit the proof.
	
	\begin{prop}\label{k1iso}
	The map $\Omega_A:K_1(A)\to K^{-1}(\C_A)$ defined by $\Omega_A([u])=[1_n,u]$ (for $n\geq1$ and a unitary $u$ in $M_n(\tilde A)$) is a natural isomorphism.
	\end{prop}
	
	\begin{definition}
Define $\Gamma_1(\phi)$ to be the set of all triples $(p,u,g)$ where $p$ is a projection in $M_\infty(\tilde A)$, $u$ is a unitary in $pM_\infty(\tilde A)p$, and $g$ is a unitary in $C([0,1])\otimes\phi(p)M_\infty(\tilde B)\phi(p)$ such that $g(0)=\phi(p)$ and $g(1)=\phi(u)$. For brevity, we will often denote these triples by the symbols $\sigma$ and $\tau$.
\begin{enumerate}[(i)]
    \item Define the direct sum operation $\oplus$ on $\Gamma_1(\phi)$ by
	\[(p,u,g)\oplus(p',u',g')=(p\oplus p',u\oplus u',g\oplus g').\]
	\item We say that two such triples $(p,u,g)$ and $(p',u',g')$ are \emph{isomorphic}, written $(p,u,g)\cong(p',u',g')$, if there is a partial isometry $v$ in $M_\infty(\tilde A)$ such that $v^*v=p$, $vv^*=p'$, $vu=u'v$, and $\phi(v)g(s)=g'(s)\phi(v)$ for all $0\leq s\leq 1$.
	\item A triple $(p,u,g)$ is called \emph{elementary} if there are homotopies $u_t$ and $g_t$ for $0\leq t\leq1$ such that $u_1=u$, $g_1=g$, $u_0=p$, $g_0(s)=\phi(p)$ for all $0\leq s\leq1$, and $g_t(1)=\phi(u_t)$ for all $0\leq t\leq1$.
	\item Two triples $\sigma$ and $\sigma'$ in $\Gamma_1(\phi)$ are \emph{equivalent}, written $\sigma\sim\sigma'$, if there exist elementary triples $\tau$ and $\tau'$ such that $\sigma\oplus\tau\cong\sigma'\oplus\tau'$.
\end{enumerate}
	   Denote by $[\sigma]$, or $[p,u,g]$, the equivalence class of the triple $\sigma=(p,u,g)$ via the relation $\sim$. $K_1(\phi)$ is then defined to be the quotient of $\Gamma_1(\phi)$ by the relation $\sim$, that is,
	\[\{[\sigma]\mid\sigma\in\Gamma_1(\phi)\}=\Gamma_1(\phi)/\sim\]
	\end{definition}
	
	It is easily checked that, like $\Gamma_0(\phi)$, the direct sum operation of triples in $\Gamma_1(\phi)$ behaves well with respect to the notions of isomorphism and elementary.
	
	\begin{prop} $K_1(\phi)$ is an abelian group when equipped with the binary operation
	\[[\sigma]+[\sigma]=[\sigma\oplus\sigma']\]
	where the identity element is given by $[0,0,0]$ and the inverse of $[p,u,g]$ is given by $[p,u^*,g^*]$.
	\end{prop}
	
	\begin{proof} We verify the last claim. We have
	\[[p,u,g]+[p,u^*,g^*]=[p\oplus p,u\oplus u^*,g\oplus g^*].\]
	Define the matrices
	\[a=\left[\begin{array}{cc}
	    0 & p \\
	    p & 0
	\end{array}\right],\qquad w=\left[\begin{array}{cc}
	    0 & u^* \\
	    u & 0
	\end{array}\right],\qquad h=\left[\begin{array}{cc}
	    0 & g^* \\
	    g & 0
	\end{array}\right].\]
	The first two are self-adjoint unitaries in $(p\oplus p)M_\infty(\tilde A)(p\oplus p)$, and the third is a self-adjoint unitary in $C([0,1])\otimes(\phi(p)\oplus\phi(p))M_\infty(\tilde B)(\phi(p)\oplus\phi(p))$. Observe that $h(1)=\phi(w)$. Define, for $0\leq t\leq1$,
	\[u_t=\exp(i\pi t(p\oplus p-a)/2)\exp(i\pi t(p\oplus p-w)/2)\]
	\[g_t=\exp(i\pi t(\phi(p)\oplus\phi(p)-\phi(a))/2)\exp(i\pi t(\phi(p)\oplus\phi(p)-h)/2)\]
	Then $u_0=p\oplus p$, $u_1=u\oplus u^*$, $g_1=g\oplus g^*$, and $g_0(s)=\phi(p)\oplus\phi(p)$ for all $0\leq s\leq1$. Moreover, $g_t(1)=\phi(u_t)$ for all $0\leq t\leq1$. It follows that $(p\oplus p,u\oplus u^*,g\oplus g^*)$ is elementary.
	\end{proof}
	
	The following result is similar to Proposition \ref{3}, so we omit the proof.
	
	\begin{prop}\label{k1prop}
	\begin{enumerate}[(i)]
	    \item Suppose we have two triples $(p,u,g)$ and $(p',u',g')$ and that $p=p'$. If $p$ is in $M_n(\tilde A)$ and $u_t$ is a path of unitaries from $u$ to $u'$ in $pM_n(\tilde A)p$ and $g_t$ is a path of unitaries from $g$ to $g'$ in $C([0,1])\otimes\phi(p)M_n(\tilde B)\phi(p)$ such that $g_t(1)=\phi(u_t)$ for all $0\leq t\leq1$, then $[p,u,g]=[p',u',g']$.
	    \item If $p=p'$, we have
	\[[p,u,g]+[p',u',g']=[p,uu',gg']=[p,u'u,g'g].\]
	\item If $(p,u,g)$ and $(p',u',g')$ are two triples in $\Gamma_1(\phi)$ such that $pp'=0$, then $(p,u,g)\oplus(p',u',g')\cong(p+p',u+u',g+g')$.
	\item Every triple in $\Gamma_1(\phi)$ is equivalent to one of the form $(1_n,u,g)$, where $n\geq1$ and $\dot u=\dot g(s)=1_n$ for all $0\leq s\leq1$.
	\item If $p$ is in $M_n(\tilde A)$, $[p,u,g]=0$ if and only if there is an integer $k\geq1$ and paths of unitaries $u_t$ in $(p\oplus1_k)M_{n+k}(\tilde A)(p\oplus1_k)$ and $g_t$ in $C([0,1])\otimes(\phi(p)\oplus1_k)M_{n+k}(\tilde B)(\phi(p)\oplus1_k)$ such that $u_0=p\oplus1_k$, $u_1=u\oplus1_k$, $g_0(s)=\phi(p)\oplus1_k$ for all $0\leq s\leq1$, $g_1=g\oplus1_k$, and $g_t(1)=\phi(u_t)$ for all $0\leq t\leq1$.
	\end{enumerate}
	\end{prop}
	
	We now collect some properties that hold for both relative groups.
	
	\begin{prop}\label{univ1} Suppose that $G$ is an abelian group and $\nu:\Gamma_j(\phi)\to G$ is a map that satisfies
	\begin{enumerate}[(i)]
	    \item $\nu(\sigma\oplus\tau)=\nu(\sigma)+\nu(\tau)$,
	    \item $\nu(\sigma)=0$ if $\sigma$ is elementary, and
	    \item if $\sigma\cong\tau$, then $\nu(\sigma)=\nu(\tau)$.
	\end{enumerate}
	Then $\nu$ factors to a unique group homomorphism $\alpha:K_j(\phi)\to G$.
	\end{prop}
	
	\begin{proof}
	If $\sigma\sim\sigma'$, find elementary triples $\tau$ and $\tau'$ such that $\sigma\oplus\tau\cong\sigma'\oplus\tau'$. Then
	\[\nu(\sigma
	)=\nu(\sigma)+\nu(\tau)=\nu(\sigma\oplus\tau)=\nu(\sigma'\oplus\tau')=\nu(\sigma')+\nu(\tau')=\nu(\sigma')\]
	So the map $\alpha([\sigma]):=\nu(\sigma)$ is well-defined. It is a group homomorphism by property (i).
	\end{proof}
	
	If $\phi:A\to B$ and $\psi:C\to D$ are $^*$-homomorphisms, we denote by $\phi\oplus\psi$ the component-wise $^*$-homomorphism $A\oplus C\to B\oplus D$.
	
	\begin{prop}\label{direct} Suppose $\phi:A\to B$ and $\psi:C\to D$ are $^*$-homomorphisms. Then there are natural isomorphisms $K_*(\phi\oplus\psi)\to K_*(\phi)\oplus K_*(\psi)$ that satisfy
	\[[(p,p'),(q,q'),(v,v')]\mapsto([p,q,v],[p',q',v'])\]
	in the case of $K_0$, and
	\[[(p,p'),(u,u'),(g,g')]\mapsto([p,u,g],[p',u',g'])\]
	in the case of $K_1$.
	\end{prop}
	
	\begin{proof}
	For a triple $((p,p'),(q,q'),(v,v'))$ in $\Gamma_0(\phi\oplus\psi)$, define
	\[\nu((p,p'),(q,q'),(v,v'))=([p,q,v],[p',q',v']).\]
	It is straightforward to check that $\nu$ satisfies the hypotheses of Proposition \ref{univ1}, so we get a well-defined group homomorphism that factors $\nu$. The fact that the group homomorphism is surjective is clear, and injectivity follows from a simple application of part (v) of Proposition \ref{3}. The proof is similar for $K_1$.
	\end{proof}
	
	\begin{prop}\label{induce}
	Suppose that
	\begin{center}
\begin{tikzcd}
A \arrow[r, "\phi"] \arrow[d, "\alpha"']   & B \arrow[d, "\beta"]  &              \\
C \arrow[r, "\psi"]              & D           
\end{tikzcd}
	\end{center}
	is a commutative diagram of C$^*$-algebras and $^*$-homomorphisms. Then there are well-defined group homomorphisms $(\alpha,\beta)_*:K_j(\phi)\to K_j(\psi)$ that satisfy
	\[(\alpha,\beta)_*([p,q,v])=[\alpha(p),\alpha(q),\beta(v)]\]
	for a triple $(p,q,v)$ in $\Gamma_0(\phi)$ and\emph{
	\[(\alpha,\beta)_*([p,u,g])=[\alpha(p),\alpha(u),\id_{C([0,1])}\otimes\beta(g)]\]}
	for a triple $(p,u,g)$ in $\Gamma_1(\phi)$. If $\alpha$ and $\beta$ are $^*$-isomorphisms, then $(\alpha,\beta)_*$ is a group isomorphism.
	\end{prop}
	
	\begin{proof}
	For a triple $(p,q,v)$ in $\Gamma_0(\phi)$, set $\nu(p,q,v)=[\alpha(p),\alpha(q),\beta(v)]$. Again, the hypotheses of Proposition \ref{univ1} are easy to check, so $\nu$ factors to a group homomorphism $(\alpha,\beta)_*$. If $\alpha$ and $\beta$ are $^*$-isomorphisms, then the diagram
	\begin{center}
\begin{tikzcd}
C \arrow[r, "\psi"] \arrow[d, "\alpha^{-1}"']   & D \arrow[d, "\beta^{-1}"] \\
                                 
A \arrow[r, "\phi"]              & B           
\end{tikzcd}
	\end{center}
	is commutative and the same argument works to obtain the group homomorphism $(\alpha^{-1},\beta^{-1})_*$, which is easily seen to be the inverse of $(\alpha,\beta)_*$. The proof is again similar for $K_1$.
	\end{proof}
	
	As an application of the above results, we will show that if $A$ and $B$ are unital and $\phi(1)=1$, one may define $K_*(\phi)$ without unitizations while remaining consistent with the results above. To verify this, let $K_*^u(\phi)$ be the group defined in the same way as $K_*(\phi)$, but avoid unitizing $A$ and $B$ and use the units already present. Notice $K_*(\phi)$ and $K_*^u(\tilde\phi)$ are precisely the same objects, and all preceding results about $K_*(\phi)$ remain true for $K_*^u(\phi)$ with appropriate modifications.
	
		\begin{prop}\label{unit} If $A$ and $B$ are unital and $\phi(1)=1$, then $K_j(\phi)$ and $K_j^u(\phi)$ are isomorphic as groups.
		\end{prop}
	
	\begin{proof} The map $\nu_A:A\oplus\bC\to\tilde A$ defined by $\nu_A(a,\lambda)=a+\lambda(1_{\tilde A}-1_A)$ is a $^*$-isomorphism and the diagram
	\begin{center}
\begin{tikzcd}
A\oplus\bC \arrow[d, "\nu_A"] \arrow[r, "\phi\oplus\id_\bC"]   & B\oplus\bC \arrow[d, "\nu_B"] \\
                   
\tilde A \arrow[r, "\tilde\phi"]     & \tilde B                       
\end{tikzcd}
	\end{center}
	is commutative. Therefore $K_j(\phi)=K_j^u(\tilde\phi)$ is isomorphic to $K_j^u(\phi\oplus\id_\bC)$ by Proposition \ref{induce}. Then
	\[K_j(\phi)=K_j^u(\tilde\phi)\cong K_j^u(\phi\oplus\id_\bC)\cong K_j^u(\phi)\oplus K_j^u(\id_\bC)\cong K_j^u(\phi)\]
	where the third isomorphism is due to Proposition \ref{direct}. The fact that $K_j^u(\id_\bC)=0$ is rather clear, but the skeptical reader is referred to part (ii) of Corollary \ref{zero}.
	\end{proof}
	
	\section{Proofs}
	
	\subsection{Proof of part (i) of Theorem \ref{main}}
	
	Define the map $\mu_0:K_1(B)\to K_0(\phi)$ by $\mu_0([u])=[1_n,1_n,u]$, where $u$ is a unitary in $M_n(\tilde B)$. By part (i) of Proposition \ref{3}, $\mu_0$ is well-defined, and clearly it is a group homomorphism.
	
	Define a map $\nu:\Gamma_0(\phi)\to K_0(A)$ by $\nu(p,q,v)=[p]-[q]$. Observe that the image of $\nu$ is indeed in $K_0(A)$ (not just $K_0(\tilde A)$) since $\dot v^*\dot v=\dot p$ and $\dot v\dot v^*=\dot q$, hence $[\dot p]=[\dot q]$. It is easy to check that $\nu$ satisfies the hypotheses of Proposition \ref{univ1}, hence factors to a well-defined group homomorphism $\nu_0:K_0(\phi)\to K_0(A)$.
    
	\begin{prop}\label{exact1} The sequence
	\begin{center}
	\begin{tikzcd}
K_1(A) \arrow[r, "\phi_*"] & K_1(B) \arrow[r, "\mu_0"] & {K_0(\phi)} \arrow[r, "\nu_0"] & K_0(A) \arrow[r, "\phi_*"] & K_0(B)
\end{tikzcd}
\end{center}
is exact.
\end{prop}
	\begin{proof} It is quite clear that all compositions are zero. If $\phi_*([p]-[q])=[\phi(p)]-[\phi(q)]=0$, choose $k\geq0$ and $w$ in $M_\infty(\tilde B)$ such that $w^*w=\phi(p)\oplus1_k$ and $ww^*=\phi(q)\oplus1_k$. Then
	\[[p]-[q]=\nu_0([p\oplus1_k,q\oplus1_k,w]),\]
	which shows exactness at $K_0(A)$.
	
	If $(p,1_n,v)$ is such that $\nu_0([p,1_n,v])=[p]-[1_n]=0$, choose $k\geq0$ and $w$ in $M_\infty(\tilde A)$ such that $w^*w=p\oplus1_k$ and $ww^*=1_n\oplus0_{m-n}\oplus1_k$. Then
	\[(p\oplus1_k,1_n\oplus0_{m-n}\oplus1_k,v\oplus1_k)\cong(1_n\oplus0_{m-n}\oplus1_k,1_n\oplus0_{m-n}\oplus1_k,(v\oplus1_k)\phi(w^*))\]
	and hence
	\[[p,1_n,v]=\mu_0([(v\oplus1_k)\phi(w^*)+0_n\oplus1_{m-n}\oplus0_k]),\]
	which shows exactness at $K_0(\phi)$.
	
	Finally, if $\mu_0([u])=[1_n,1_n,u]=0$, use part (v) of Proposition \ref{3} to find $k\geq0$ and a partial isometry $w$ such that $\phi(w)(u\oplus1_k)$ is a unitary and homotopic to $1_{n+k}$ in $M_{n+k}(\tilde B)$. Since $u\oplus1_k$ is a unitary, so is $w$ and $u\oplus1_k$ is homotopic to $\phi(w^*)$. Thus
	\[[u]=[u\oplus1_k]=[\phi(w^*)]=\phi_*([w^*]),\]
	which shows exactness at $K_1(B)$.
	\end{proof}
	
	For a unitary $g$ in $C([0,1])\otimes M_n(\tilde B)$ with $g(0)=g(1)=\dot g=1_n$, set $\mu_1([g])=[1_n,1_n,g]$. By part (i) of Proposition \ref{k1prop}, this is a well-defined group homomorphism $\mu_1:K_1(SB)\to K_1(\phi)$. For a triple $(p,u,g)$ in $\Gamma_1(\phi)$, define $\nu(p,u,g)=[p,u]$ (here we use the picture of $K_1$ described before Proposition \ref{k1iso}). The hypotheses of Proposition \ref{univ1} are satisfied, so we get a group homomorphism $\nu_1:K_1(\phi)\to K_1(A)$ such that $\nu_1([p,u,g])=[p,u]$. If $p=1_n$, the formula is more simply $\nu_1([1_n,u,g])=[u]$.
	
	\begin{prop}\label{exact2} The sequence
		\begin{center}
	\begin{tikzcd}
K_1(SA) \arrow[r, "(S\phi)_*"] & K_1(SB) \arrow[r, "\mu_1"] & {K_1(\phi)} \arrow[r, "\nu_1"] & K_1(A) \arrow[r, "\phi_*"] & K_1(B)
\end{tikzcd}
\end{center}
	is exact.
	\end{prop}
	
	\begin{proof} Again, all compositions are clearly zero. If $\phi_*([u])=0$, we may find $k\geq0$ and a unitary $g$ in $C([0,1])\otimes M_{n+k}(\tilde B)$ such that $g(1)= \phi(u)\oplus1_k$ and $g(0)=1_{n+k}$. Then
	\[[u]=\nu_1([1_{n+k},u\oplus1_k,g]),\]
	which shows exactness at $K_1(A)$.
	
	If $\nu_1([1_n,u,g])=[u]=0$, find $k\geq0$ and a unitary $f$ in $C([0,1])\otimes M_{n+k}(\tilde A)$ such that $f(0)=1_{n+k}$ and $f(1)=u\oplus1_k$. Set
	\[\tilde g(s)=\left\{\begin{array}{cc}
	    g(2s)\oplus1_k & 0\leq s\leq1/2 \\
	    \phi(f(2-2s)) & 1/2\leq s\leq1
	\end{array}\right.\]
	Then $\tilde g$ is a unitary in $C([0,1])\otimes M_{n+k}(\tilde B)$ and $\tilde g(0)=\tilde g(1)=1_{n+k}$. Now for a fixed $t$ in $[0,1]$, the function $g_t$ defined by
	\[g_t(s)=\left\{\begin{array}{cc}
	    g(s(1-\frac12t)^{-1})\oplus1_k & 0\leq s\leq1-\frac12t \\
	    \phi(f(3-2s-t)) & 1-\frac12t\leq s\leq1
	\end{array}\right.\]
	satisfies $g_0=g\oplus1_k$, $g_1=\tilde g$, and $g_t(1)=\phi(f(1-t))$, and so
	\[[1_n,u,g]=[1_{n+k},u\oplus1_k,g\oplus1_k]=[1_{n+k},1_{n+k},\tilde g]=\mu_1([\tilde g]),\]
	which shows exactness at $K_1(\phi)$.
	
	Finally, if $\mu_1([g])=[1_n,1_n,g]=0$, use part (v) of Proposition \ref{k1prop} to find an integer $k$ and homotopies $u_t$ and $g_t$ such that $u_0=u_1=1_{n+k}$, $g_1=g\oplus1_k$, $g_0=1_{n+k}$, and $g_t(1)=\phi(u_t)$ for all $t$. Write $f(t)=u_t$ and set
	\[\tilde g_t(s)=\left\{\begin{array}{cc}
	    g_t(2s) & 0\leq s\leq1/2 \\
	    \phi(f((2-2t)s+2t-1)) & 1/2\leq s\leq1
	\end{array}\right.\]
	Then $\tilde g_t(0)=\tilde g_t(1)=1_{n+k}$ for all $t$ and
	\[\tilde g_1(s)=\left\{\begin{array}{cc}
	    g(2s)\oplus1_k & 0\leq s\leq1/2 \\
	    1_{n+k} & 1/2\leq s\leq1
	\end{array}\right.\]
	and
	\[\tilde g_0(s)=\left\{\begin{array}{cc}
	    1_{n+k} & 0\leq s\leq1/2 \\
	    \phi(f(2s-1)) & 1/2\leq s\leq1
	\end{array}\right.\]
	which are homotopic to $g\oplus1_k$ and $S\phi(f)$, respectively. Thus
	\[[g]=[g\oplus1_k]=[S\phi(f)]=(S\phi)_*([f]),\]
	which shows exactness at $K_1(SB)$.
	\end{proof}
	
	\begin{prop} If $\phi=0$, then the sequence in part (i) of Theorem \ref{main} splits at $K_0(A)$ and $K_1(A)$. In other words, for each $j=0,1$ there is a group homomorphism $\lambda_j:K_j(A)\to K_j(\phi)$ such that $\nu_j\circ\lambda_j$ is the identity map on $K_j(A)$.
	\end{prop}
	
	\begin{proof} If $p$ and $q$ are two projections in $M_\infty(\tilde A)$ with $[\dot p]=[\dot q]$, let $v$ be a partial isometry in $M_\infty(\bC)$ such that $v^*v=\dot p$ and $vv^*=\dot q$. If $u$ is a unitary in $M_n(\tilde A)$, let $g$ be any unitary in $C([0,1])\otimes M_n(\bC)$ such that $g(0)=1_n$ and $g(1)=\dot u$. Define \[\lambda_0([p]-[q])=[p,q,v]\qquad\lambda_1([u])=[1_n,u,g]\]
	For both $j=0,1$, it is straightforward to check that $\lambda_j$ is well-defined, additive, independent of the choices of $v$ and $g$, and that $\nu_j\circ\lambda_j$ is the identity.
	\end{proof}
	
	By combining all results in this subsection, we obtain part (i) of Theorem \ref{main}. The map $\mu_1$ is (by abuse of notation) the composition of the Bott map $\beta_B$ and $\mu_1$ from Proposition \ref{exact2}. It may therefore be written, for projections $p$ and $q$ in $M_n(\tilde B)$, as $\mu_1([p]-[q])=[1_n,1_n,f_pf_q^*]$, where $f_p(s)=e^{2\pi isp}$ for $0\leq s\leq1$. Since the Bott map is natural, this does not affect exactness.
	
	We also record the following immediate and useful consequences of part (i) of Theorem \ref{main}.
	
	\begin{corollary}\label{zero} We have the following. \begin{enumerate}[(i)]
	    \item If $K_*(A)=K_*(B)=0$, then $K_*(\phi)=0$.
	    \item If $\phi:A\to B$ is a $^*$-isomorphism, then $K_*(\phi)=0$.
	\end{enumerate}
	\end{corollary}
	
	\subsection{Proof of parts (ii) and (iii) of Theorem \ref{main}}
	
	Throughout this subsection, we will assume that
	\begin{equation}\label{comm1}
\begin{tikzcd}
0 \arrow[r] & I \arrow[r, "\iota_A"] \arrow[d, "\psi"] & A \arrow[r, "\pi_A"] \arrow[d, "\phi"] & A/I \arrow[r] \arrow[d, "\gamma"] & 0 \\
0 \arrow[r] & J \arrow[r, "\iota_B"]                     & B \arrow[r, "\pi_B"]                   & B/J \arrow[r]                    & 0
\end{tikzcd}
\end{equation}
is a commutative diagram with exact rows. We will abbreviate the induced maps $(\iota_A,\iota_B)_*$ and $(\pi_A,\pi_B)_*$ to $\iota_*$ and $\pi_*$, respectively.

	\begin{prop}\label{half} The sequence
\begin{center}
\begin{tikzcd}
K_0(\psi) \arrow[r, "\iota_*"] & K_0(\phi) \arrow[r, "\pi_*"] & K_0(\gamma)
\end{tikzcd}
\end{center}
is exact. If $\lambda_A:A/I\to A$ and $\lambda_B:B/J\to B$ are splittings of the rows in (\ref{comm1}) that keep the diagram commutative, then the sequence
\begin{center}
\begin{tikzcd}
0 \arrow[r] & K_0(\psi) \arrow[r, "\iota_*"] & K_0(\phi) \arrow[r, shift left, "\pi_*"] & K_0(\gamma) \arrow[l, shift left, "\lambda_*"] \arrow[r] & 0
\end{tikzcd}
\end{center}
is split exact, where $\lambda_*=(\lambda_A,\lambda_B)_*$.
\end{prop}

\begin{proof} It is clear that the composition is zero. Conversely, suppose that $[1_n,q,v]$ is in the kernel of $\pi_*$, so $[1_n,\pi_A(q),\pi_B(v)]=0$. Find (in order):
\begin{enumerate}[(i)]
    \item an integer $m\geq n$ so that $q$ is in $M_m(\tilde A)$,
    \item an integer $k\geq0$ and a partial isometry $w$ in $M_{m+k}(\widetilde{A/I})$ such that $w^*w=\pi_A(q)\oplus1_k$ and $ww^*=1_n\oplus0_{m-n}\oplus1_k$ and $\gamma(w)(\pi_B(v)\oplus1_k)$ is homotopic to $1_n\oplus0_{m-n}\oplus1_k$ (use part (v) of Proposition \ref{3}),
    \item an integer $l\geq0$ and a unitary $z$ homotopic to $1_{m+k+l}$ in $M_{m+k+l}(\widetilde{A/I})$ such that $z(\pi_A(q)\oplus1_k\oplus0_l)z^*=1_n\oplus0_{m-n}\oplus1_k\oplus0_l$ and $\gamma(z)(\pi_B(v)\oplus1_k\oplus0_l)=(\gamma(w)(\pi_B(v)\oplus1_k))\oplus0_l$. For example, one may take $l=m+k$ and
\[z=\left[\begin{array}{cc}
    w & 1_{m+k}-ww^* \\
    1_{m+k}-w^*w & w^*
\end{array}\right],\]
see the discussion following Definition \ref{projpath}.
\item a unitary $U$ in $M_{m+k+l}(\tilde A)$ such that $\pi_A(U)=z$ (this is possible because $z$ is homotopic to $1_{m+k+l}$),
\item a unitary $V$ in $(1_n\oplus0_{m-n}\oplus1_k)M_{m+k}(\tilde B)(1_n\oplus0_{m-n}\oplus1_k)$ homotopic to $1_n\oplus0_{m-n}\oplus1_k$ such that $\pi_B(V)=\gamma(w)(\pi_B(v)\oplus1_k)$ (use (ii)).
\end{enumerate}
Then
\begin{align}
\nonumber[1_n,q,v]&=\left[\left[\begin{array}{cccc}
    1_n & 0 & 0 & 0 \\
    0 & 0_{m-n} & 0 & 0 \\
    0 & 0 & 1_k & 0 \\
    0 & 0 & 0 & 0_l
    \end{array}\right],\left[\begin{array}{ccc}
    q & 0 & 0 \\
    0 & 1_k & 0 \\
    0 & 0 & 0_l
    \end{array}\right],\left[\begin{array}{ccc}
    v & 0 & 0 \\
    0 & 1_k & 0 \\
    0 & 0 & 0_l
    \end{array}\right]\right]\\
\nonumber&=\left[\left[\begin{array}{cccc}
    1_n & 0 & 0 & 0 \\
    0 & 0_{m-n} & 0 & 0 \\
    0 & 0 & 1_k & 0 \\
    0 & 0 & 0 & 0_l
    \end{array}\right],U\left[\begin{array}{ccc}
    q & 0 & 0 \\
    0 & 1_k & 0 \\
    0 & 0 & 0_l
    \end{array}\right]U^*,\phi(U)\left[\begin{array}{ccc}
    v & 0 & 0 \\
    0 & 1_k & 0 \\
    0 & 0 & 0_l
    \end{array}\right]\right]\\
\nonumber&=\left[\left[\begin{array}{cccc}
    1_n & 0 & 0 & 0 \\
    0 & 0_{m-n} & 0 & 0 \\
    0 & 0 & 1_k & 0 \\
    0 & 0 & 0 & 0_l
    \end{array}\right],U\left[\begin{array}{ccc}
    q & 0 & 0 \\
    0 & 1_k & 0 \\
    0 & 0 & 0_l
    \end{array}\right]U^*,\phi(U)\left[\begin{array}{ccc}
    v & 0 & 0 \\
    0 & 1_k & 0 \\
    0 & 0 & 0_l
    \end{array}\right]\right]\\
\nonumber&\qquad+\left[\left[\begin{array}{cccc}
    1_n & 0 & 0 & 0 \\
    0 & 0_{m-n} & 0 & 0 \\
    0 & 0 & 1_k & 0 \\
    0 & 0 & 0 & 0_l
    \end{array}\right],\left[\begin{array}{cccc}
    1_n & 0 & 0 & 0 \\
    0 & 0_{m-n} & 0 & 0 \\
    0 & 0 & 1_k & 0 \\
    0 & 0 & 0 & 0_l
    \end{array}\right],\left[\begin{array}{cc}
    V^* & 0 \\
    0 & 0_l
    \end{array}\right]\right]\\
\nonumber&=\left[\left[\begin{array}{cccc}
    1_n & 0 & 0 & 0 \\
    0 & 0_{m-n} & 0 & 0 \\
    0 & 0 & 1_k & 0 \\
    0 & 0 & 0 & 0_l
    \end{array}\right],U\left[\begin{array}{ccc}
    q & 0 & 0 \\
    0 & 1_k & 0 \\
    0 & 0 & 0_l
    \end{array}\right]U^*,\phi(U)\left[\begin{array}{ccc}
    v & 0 & 0 \\
    0 & 1_k & 0 \\
    0 & 0 & 0_l
    \end{array}\right]\left[\begin{array}{cc}
        V^* & 0 \\
        0 & 0_l
    \end{array}\right]\right]
\end{align}
To get the first equality above, we added an elementary scalar triple. To get the second, notice that the two triples are isomorphic via the unitary $U$. In the third equality, the new triple being added is elementary because $V$ is homotopic to the identity. The fourth equality follows from part (ii) of Proposition \ref{3}. Regarding the elements of the latter triple, we have
\[\pi_A\left(U\left[\begin{array}{ccc}
    q & 0 & 0 \\
    0 & 1_k & 0 \\
    0 & 0 & 0_l
    \end{array}\right]U^*\right)=\pi_B\left(\phi(U)\left[\begin{array}{ccc}
    v & 0 & 0 \\
    0 & 1_k & 0 \\
    0 & 0 & 0_l
    \end{array}\right]\left[\begin{array}{cc}
    V^* & 0 \\
    0 & 0_l
    \end{array}\right]\right)=\left[\begin{array}{cccc}
    1_n & 0 & 0 & 0 \\
    0 & 0_{m-n} & 0 & 0 \\
    0 & 0 & 1_k & 0 \\
    0 & 0 & 0 & 0_l
    \end{array}\right]\]
from which it follows that $[1_n,q,v]$ is in the image of $\iota_*$.

For the split exact sequence, it is clear that $\lambda_*$ is a right inverse for $\pi_*$, so we need only show that $\iota_*$ is injective. Suppose that $(1_n,q,v)$ is a triple in $\Gamma_0(\psi)$ with $\dot q=\dot v=1_n$ and $[1_n,q,v]=0$ in $K_0(\phi)$. Choose $m\geq n$ so that $1_n\oplus0_{m-n}$ and $q$ are in $M_m(\tilde I)$ and $v$ is in $M_m(\tilde J)$. Use part (v) of Proposition \ref{3} to find an integer $k\geq0$ and a partial isometry $w$ in $M_{m+k}(\tilde A)$ with $w^*w=q\oplus1_k$ and $ww^*=1_n\oplus0_{m-n}\oplus1_k$ and $\phi(w)(v\oplus1_k)$ is homotopic to $1_n\oplus0_{m-n}\oplus1_k$. Let $y_t$ be such a homotopy, that is, $y_0=\dot y_t=1_n\oplus0_{m-n}\oplus1_k$ for all $t$ and $y_1=\phi(w)(v\oplus1_k)$. Set $x=\lambda_A(\pi_A(w^*))w$. Then $\pi_A(x)=1_n\oplus0_{m-n}\oplus1_k$ so that $x$ is in $M_{m+k}(\tilde I)$. We have $x^*x=q\oplus1_k$ and $xx^*=1_n\oplus0_{m-n}\oplus1_k$ and, since $\pi_B(v\oplus1_k)=1_n\oplus0_{m-n}\oplus1_k$,
\begin{align}\nonumber\psi(x)\left[\begin{array}{cc}
    v & 0 \\
    0 & 1_k
\end{array}\right]&=\psi(\lambda_A(\pi_A(w^*))w)\left[\begin{array}{cc}
    v & 0 \\
    0 & 1_k
\end{array}\right]\\
\nonumber&=\lambda_B(\pi_B(\phi(w^*)))\phi(w)\left[\begin{array}{cc}
    v & 0 \\
    0 & 1_k
\end{array}\right]\\
\nonumber&=\lambda_B\left(\pi_B\left(\left[\begin{array}{cc}
    v & 0 \\
    0 & 1_k
\end{array}\right]\phi(w^*)\right)\right)\phi(w)\left[\begin{array}{cc}
    v & 0 \\
    0 & 1_k
\end{array}\right]\\
\nonumber&=\lambda_B\left(\pi_B\left(\phi(w)\left[\begin{array}{cc}
    v & 0 \\
    0 & 1_k
\end{array}\right]\right)^*\right)\phi(w)\left[\begin{array}{cc}
    v & 0 \\
    0 & 1_k
\end{array}\right]
\end{align}
is homotopic to $1_n\oplus0_{m-n}\oplus1_k$ through $M_{m+k}(\tilde J)$ via $\lambda_B(\pi_B(y_t^*))y_t$. It follows that $[1_n,q,v]=0$ in $K_0(\psi)$.
\end{proof}
	
	Now we associate an index map $\partial_1:K_1(\gamma)\to K_0(\psi)$ to the diagram (\ref{comm1}).
	
	\begin{definition}\label{index}
	The index map $\partial_1:K_1(\gamma)\to K_0(\psi)$ is given by
	\[\partial_1([1_n,u,g])=\left[w\left[\begin{array}{cc}
	    1_n & 0 \\
	    0 & 0_l
	\end{array}\right]w^*,\left[\begin{array}{cc}
	    1_n & 0 \\
	    0 & 0_l
	\end{array}\right],\left[\begin{array}{cc}
	    h(1) & 0 \\
	    0 & 0_l
	\end{array}\right]\phi(w^*)\right]\]
	where $l\geq0$, $w$ is a unitary in $M_{n+l}(\tilde A)$ such that $\pi_A(w)(1_n\oplus0_l)=u\oplus0_l$, and $h$ is a unitary in $C([0,1])\otimes M_n(\widetilde{B})$ such that $h(0)=1_n$ and $\pi_B(h)=g$.
	\end{definition}
	
	Observe that such elements $l$, $w$, and $h$ always exist: one may take $l=n$, $w$ to be a lift of $u\oplus u^*$, and $h$ exists because $g$, as a unitary in $C([0,1])\otimes M_n(\widetilde{B/J})$ is homotopic to $1_n$. It is straightforward to verify that $\partial_1$ is independent of these choices, and depends only on the class of the triple $(1_n,u,g)$.
	
	The map $\partial_1$ is natural in the following sense. Suppose that
	\small
	\begin{center}
\begin{tikzcd}
             & 0 \arrow[rr] &                                                      & I' \arrow[rr] \arrow[dd, "\psi'" near start] &                                            & A' \arrow[rr] \arrow[dd, "\phi'" near start] &                                                        & A'/I' \arrow[rr] \arrow[dd, "\gamma'" near start] &   & 0 \\
0 \arrow[rr] &              & I \arrow[rr] \arrow[ru, "\sigma"] \arrow[dd, "\psi" near start] &                                   & A \arrow[rr] \arrow[dd, "\phi" near start] \arrow[ru] &                                   & A/I \arrow[rr] \arrow[dd, "\gamma" near start] \arrow[ru, "\tau"] &                                       & 0 &   \\
             & 0 \arrow[rr] &                                                      & J' \arrow[rr]                     &                                            & B' \arrow[rr]                     &                                                        & B'/J' \arrow[rr]                      &   & 0 \\
0 \arrow[rr] &              & J \arrow[rr] \arrow[ru, "\sigma'"]                   &                                   & B \arrow[rr] \arrow[ru]                    &                                   & B/J \arrow[rr] \arrow[ru, "\tau'"]                     &                                       & 0 &  
\end{tikzcd}
	\end{center}
	\normalsize
    is a commutative diagram with exact rows. Then the diagram
	
	\begin{center}
\begin{tikzcd}
K_1(\gamma) \arrow[rr, "\partial_1"] \arrow[dd, "{(\tau,\tau')_*}"] &  & K_0(\psi) \arrow[dd, "{(\sigma,\sigma')_*}"] \\
                                                          &  &                                  \\
K_1(\gamma') \arrow[rr, "\partial_1'"]                    &  & K_0(\psi')                      
\end{tikzcd}
	\end{center}
	is commutative. We leave the straighforward proof to the reader.
	
	\begin{prop}\label{connect} The sequence
    \begin{center}
\begin{tikzcd}
K_1(\phi) \arrow[r, "\pi_*"] & K_1(\gamma) \arrow[r, "\partial_1"] & K_0(\psi) \arrow[r, "\iota_*"] & K_0(\phi)
\end{tikzcd}
    \end{center}
    is exact and the diagram
    \begin{center}
\begin{tikzcd}
K_1(S(B/J)) \arrow[dd, "\theta_J^{-1}\circ\delta_2"] \arrow[rr, "\mu_1"] &  & K_1(\gamma) \arrow[rr, "\nu_1"] \arrow[dd, "\partial_1"] &  & K_1(A/I) \arrow[dd, "\delta_1"] \\
                                                          &  &                                                         &  &                                 \\
K_1(J) \arrow[rr, "\mu_0"]                           &  & K_0(\psi) \arrow[rr, "\nu_0"]                       &  & K_0(I)                         
\end{tikzcd}
    \end{center}
    is commutative.
	\end{prop}
	
	\begin{proof}
	For ease of notation we will denote
	\[p=w\left[\begin{array}{cc}
	    1_n & 0 \\
	    0 & 0_l
	\end{array}\right]w^*\qquad v=\left[\begin{array}{cc}
	    h(1) & 0 \\
	    0 & 0_l
	\end{array}\right]\phi(w^*)\]
	It is a simple calculation to see that the right square in the diagram is commutative. For the left square, take $[f]$ in $K_1(S(B/J))$, where $f$ is in $M_n(\widetilde{S(B/J)})$ and $f(0)=1_n$. Find $h$ in $M_n(\widetilde{CB})$ such that $h(0)=1_n$ and $\pi_B(h)=f$. Then
	\[\partial_1(\mu_1([f]))=\partial_1([1_n,1_n,f])=[1_n,1_n,h(1)]\]
	Now find $g$ in $M_{2n}(\widetilde{SB})$ such that $g(0)=1_{2n}$ and $\pi_B(g)=f\oplus f^*$. Let
	\[\tilde g(t)=\left\{\begin{array}{cc}
	    g(2t) & 0\leq t\leq1/2 \\
	    h(2t-1)\oplus h(2t-1)^* & 1/2\leq t\leq1
	\end{array}\right.\]
	Then $[g(1_n\oplus0_n)g^*]-[1_n\oplus0_n]=[\tilde g(1_n\oplus0_n)\tilde g^*]-[1_n\oplus0_n]$ in $K_0(SJ)$, the latter being equal to $\theta_J([h(1)])$ since $\tilde g(1)=h(1)\oplus h(1)^*$. All in all, we have
	\[\mu_0(\theta_J^{-1}(\delta_2([f])))=\mu_0(\theta_J^{-1}([\tilde g(1_n\oplus0_n)\tilde g^*]-[1_n\oplus0_n]))=\mu_0([h(1)])=[1_n,1_n,h(1)]\]
	which shows commutativity of the left square.
	
	The composition $\partial_1\circ\pi_*$ is clearly zero since everything has a unitary lift. We also have $\iota_*\circ\partial_1$ zero since
	\[[p,1_n\oplus0_l,v]=[p,1_n\oplus0_l,v]+[1_n\oplus0_l,1_n\oplus0_l,h(1)^*\oplus0_l]=[p,1_n\oplus0_l,(1_n\oplus0_l)\phi(w^*)]\]
	Because $(1_n\oplus0_l,1_n\oplus0_l,h(1)^*\oplus0_l)$ is elementary in $\Gamma_0(\phi)$ and $(p,1_n\oplus0_l,(1_n\oplus0_l)\phi(w^*))\cong(1_n\oplus0_l,1_n\oplus0_l,1_n\oplus0_l)$.
	
	Now suppose that
	\[\partial_1([1_n,u,g])=[p,1_n\oplus0_l,v]=[w(1_n\oplus0_l)w^*,1_n\oplus0_l,(h(1)\oplus0_l)\phi(w^*)]=0\]
	Find $k\geq1$ and a partial isometry $x$ in $M_{n+l+k}(\tilde I)$ with $xx^*=p\oplus1_k$ and $\dot x=x^*x=1_n\oplus0_l\oplus1_k$, and such that $(v\oplus1_k)\psi(x)$ is homotopic to $1_n\oplus0_l\oplus1_k$. Let $y_t$ be such a homotopy, with $\dot y_t=y_0=1_n\oplus0_l\oplus1_k$ for all $t$ and $y_1=(v\oplus1_k)\psi(x)$. Set
	\[z=\left[\begin{array}{ccc}
	    1_n & 0 & 0 \\
	    0 & 0_l & 0 \\
	    0 & 0 & 1_k
	\end{array}\right]\left[\begin{array}{cc}
	    w^* & 0 \\
	    0 & 1_k
	\end{array}\right]x\]
	and
	\[h'(t)=\left\{\begin{array}{cc}
	    y_{2t} & 0\leq t\leq1/2 \\
	    (h(2t-1)^*\oplus0_l\oplus1_k)(v\oplus1_k)\psi(x) & 1/2\leq t\leq1
	\end{array}\right.\]
	Then $\pi_A(z)=u\oplus0_l\oplus1_k$ and
	\[\pi_B(h'(t))=\left\{\begin{array}{cc}
	    1_n\oplus0_l\oplus1_k & 0\leq t\leq1/2 \\
	    g(2t-1)\oplus0_l\oplus1_k & 1/2\leq t\leq1
	\end{array}\right.\]
	which is clearly homotopic to $g\oplus0_l\oplus1_k$. Moreover, $h'(1)=\phi(z)$. It follows that
	\[[1_n,u,g]=[1_n\oplus0_l\oplus1_k,u\oplus0_l\oplus1_k,g\oplus0_l\oplus1_k]=\pi_*([1_n\oplus0_l\oplus1_k,z,h'])\]
	Now suppose that $(p,1_n,v)$ is a triple in $\Gamma_0(\psi)$ with $[p,1_n,v]=0$ in $K_0(\phi)$. Choose $m\geq n$ such that $1_n\oplus0_{m-n}$ and $p$ are in $M_m(\tilde I)$ and $v$ is in $M_m(\tilde J)$. Find $k\geq0$ and a partial isometry $x$ in $M_{m+k}(\tilde A)$ with $xx^*=p\oplus1_k$ and $\dot x=x^*x=1_n\oplus0_{m-n}\oplus1_k$, and such that $(v\oplus1_k)\phi(x)$ is homotopic to $1_n\oplus0_{m-n}\oplus1_k$. Find a unitary $U$ in $M_{m+k}(\bC)$ such that
	\[U(1_n\oplus0_{m-n}\oplus1_k)U^*=1_{n+k}\oplus0_{m-n}\]
	and let $p'=U(p\oplus1_k)U^*$, $v'=U(v\oplus1_k)U^*$, and $x'=UxU^*$. Clearly $(p,1_n,v)\oplus(1_k,1_k,1_k)\cong(p',1_{n+k},v')$, $x'x'^*=p'$, $x'^*x'=1_{n+k}\oplus0_{m-n}$, and that $v'\phi(x')$ is homotopic to $1_{n+k}\oplus0_{m-n}$. Let $y_t$ be such a homotopy, with $\dot y_t=y_0=1_{n+k}\oplus0_{m-n}$ for all $t$ and $y_1=v'\phi(x')$. Notice that $\pi_A(x')=(1_{n+k}\oplus0_{m-n})\pi_A(x')(1_{n+k}\oplus0_{m-n})$, so we may regard $\pi_A(x')$ as a unitary in $M_{n+k}(\widetilde{A/I})$, and similarly we may regard $y_t$ as a path of unitaries in $M_{n+k}(\tilde B)$. Set $g(t)=\pi_B(y_t)$ and notice that
	\[g(1)=\pi_B(\phi(x'^*))\pi_B(v'^*)=\gamma(\pi_A(x'^*))\]
	so that $(1_{n+k},\pi_A(x'^*),g)$ is a triple in $\Gamma_1(\gamma)$. Moreover, we see that its image under $\partial_1$ is $[p,1_n,v]$ by using $l=2m+k-n$,
	\[w=\left[\begin{array}{cc}
	    x' & 1_{m+k}-x'x'^* \\
	    1_{m+k}-x'^*x' & x'^*
	\end{array}\right]\]
	in $M_{2(m+k)}(\tilde A)$ and $h(t)=y_t$.
	\end{proof}
	
	\begin{corollary} There is an isomorphism $\theta_\phi:K_1(\phi)\to K_0(S\phi)$. Moreover, the diagram
	\begin{center}
\begin{tikzcd}
K_1(SB) \arrow[rr, "\mu_1"] \arrow[dd, equal] &  & K_1(\phi) \arrow[rr, "\nu_1"] \arrow[dd, "\theta_\phi"] &  & K_1(A) \arrow[dd, "\theta_A"] \\
                              &  &                                 &  &                   \\
K_1(SB) \arrow[rr, "\mu_0"]            &  & K_0(S\phi) \arrow[rr, "\nu_0"]           &  & K_0(SA)          
\end{tikzcd}
	\end{center}
	is commutative.
	\end{corollary}
	
	\begin{proof} The map $\theta_\phi$ is the index map $\partial_1$ associated to the commutative diagram
	\begin{center}
\begin{tikzcd}
0 \arrow[r] & SA \arrow[r] \arrow[d, "S\phi"] & CA \arrow[r] \arrow[d, "C\phi"] & A \arrow[r] \arrow[d, "\phi"] & 0 \\
0 \arrow[r] & SB \arrow[r]                    & CB \arrow[r]                    & B \arrow[r]                   & 0
\end{tikzcd}
	\end{center}
	$CA$ and $CB$ are contractible, hence the relative groups $K_0(C\phi)$ and $K_1(C\phi)$ are trivial by Corollary \ref{zero}. It follows that $\theta_\phi$ is an isomorphism.
	\end{proof}
	
	An explicit description of $\theta_\phi$ is as follows. Let $(1_n,u,g)$ be a triple in $\Gamma_1(\phi)$, and let $w$ be a unitary in $C([0,1])\otimes M_{2n}(\tilde A)$ with $w(0)=1_{2n}$ and $w(1)=u\oplus u^*$. Then
	\[\theta_\phi([1_n,u,g])=\left[w\left[\begin{array}{cc}
	    1_n & 0 \\
	    0 & 0_n
	\end{array}\right]w^*,\left[\begin{array}{cc}
	    1_n & 0 \\
	    0 & 0_n
	\end{array}\right],\left[\begin{array}{cc}
	    g & 0 \\
	    0 & 0_n
	\end{array}\right]\phi(w^*)\right]\]
	
	\begin{corollary}
	The sequence
\begin{center}
\begin{tikzcd}
K_1(\psi) \arrow[r, "\iota_*"] & K_1(\phi) \arrow[r, "\pi_*"] & K_1(\gamma)
\end{tikzcd}
\end{center}
is exact. If $\lambda_A:A/I\to A$ and $\lambda_B:B/J\to B$ are splittings of the rows in (\ref{comm1}) that keep the diagram commutative, then the sequence
\begin{center}
\begin{tikzcd}
0 \arrow[r] & K_1(\psi) \arrow[r, "\iota_*"] & K_1(\phi) \arrow[r, shift left, "\pi_*"] & K_1(\gamma) \arrow[l, shift left, "\lambda_*"] \arrow[r] & 0
\end{tikzcd}
\end{center}
is split exact.
	\end{corollary}
	
	\begin{proof}
	The map $\theta_\phi$ is natural, so we have the commutative diagram
	
	\begin{center}
\begin{tikzcd}
K_1(\psi) \arrow[dd, "\theta_\psi"] \arrow[rr, "\iota_*"] &  & K_1(\phi) \arrow[rr, "\pi_*"] \arrow[dd, "\theta_\phi"] &  & K_1(\gamma) \arrow[dd, "\theta_\gamma"] \\
                                                          &  &                                                         &  &                                         \\
K_0(S\psi) \arrow[rr, "\iota_*"]                          &  & K_0(S\phi) \arrow[rr, "\pi_*"]                          &  & K_0(S\gamma)                           
\end{tikzcd}
	\end{center}
	in which, by Proposition \ref{half}, the bottom row is exact. It follows that the top row is exact as well. The proof for split exactness is similar.
	\end{proof}
	
	At this point we may unambiguously define higher relative groups $K_j(\phi)$ by $K_0(S^j\phi)$ and higher index maps $\partial_j:K_j(\gamma)\to K_{j-1}(\psi)$ to obtain the long exact sequence in part (ii) of Theorem \ref{main}. We proceed to prove that Bott periodicity holds so that the long exact sequence collapses to the six-term exact sequence in part (iii) of Theorem \ref{main}.
	
	For Bott periodicity we will follow the original proof in \cite{cuntz}. Recall that the \emph{Toeplitz algebra} $\T$ is the universal C$^*$-algebra generated by an isometry. Let $\pi:\T\to C(\bT)$ be the $^*$-homomorphism that sends the generating isometry to the function $z$ on $\bT$. The kernel of $\pi$ is isomorphic to $\K$, and by identifying $C_0((0,1))$ with elements in $C(\bT)$ that vanish at $1$ and letting $\T_0=\pi^{-1}(C_0((0,1)))$, we obtain the short exact sequence
	
	\begin{center}
	    \begin{tikzcd}
0 \arrow[r] & \K \arrow[r, hook] & \T_0 \arrow[r, "\pi"] & C_0((0,1)) \arrow[r] & 0
\end{tikzcd}
	\end{center}
	
	We will assume the nontrivial fact that $K_*(\T_0)=0$; we refer the reader to \cite{cuntz} for the original proof.
	\begin{lemma}\label{nuclear} If $C$ is in the bootstrap category (22.3.4 of \cite{blackadar}) and $K_*(C)=0$, then \emph{$K_*(\phi\otimes\id_C)=0$.} In particular, \emph{$K_*(\phi\otimes\id_{\T_0})=0$.}
	\end{lemma}
	
	\begin{proof} By the K\"unneth Theorem for tensor products (see the main result of \cite{sch}), we have $K_*(A\otimes C)=K_*(B\otimes C)=0$. The conclusion follows from Corollary \ref{zero}.
	\end{proof}
	
	\begin{lemma}\label{stable}\emph{ $K_j(\phi\otimes\id_\K)\cong K_j(\phi)$ for $j=0,1$.}
	\end{lemma}
	
	\begin{proof}
	For either $j=0,1$, we have a commutative diagram with exact rows
	
	\begin{center}
\begin{tikzcd}
K_{1-j}(A) \arrow[dd] \arrow[r] & K_{1-j}(B) \arrow[r] \arrow[dd] & K_j(\phi) \arrow[r] \arrow[dd]   & K_j(A) \arrow[r] \arrow[dd] & K_j(B) \arrow[dd] \\
                            &                             &                                  &                             &                   \\
K_{1-j}(A\otimes\K) \arrow[r]   & K_{1-j}(B\otimes\K) \arrow[r]   & K_j(\phi\otimes\id_\K) \arrow[r] & K_j(A\otimes\K) \arrow[r]   & K_j(B\otimes\K)  
\end{tikzcd}
	\end{center}
	where all the vertical maps are induced by the embedding $a\mapsto a\otimes p$, where $p$ is any rank one projection in $\K$. All vertical maps except for the middle one are known to be isomorphisms. The five lemma then shows that the middle vertical arrow is an isomorphism.
	\end{proof}
	
	We now produce the Bott map. We have the commutative diagram
	
	\begin{center}
\begin{tikzcd}
0 \arrow[r] & A\otimes\K \arrow[r] \arrow[d, "\phi\otimes\id_\K"] & A\otimes\T_0 \arrow[r] \arrow[d, "\phi\otimes\id_{\T_0}"] & SA \arrow[r] \arrow[d, "S\phi"] & 0 \\
0 \arrow[r] & B\otimes\K \arrow[r]           & B\otimes\T_0 \arrow[r]           & SB \arrow[r]           & 0
\end{tikzcd}
	\end{center}
	Proposition \ref{connect} implies that
	
	\begin{center}
	    \begin{tikzcd}
K_1(\phi\otimes\id_{\T_0}) \arrow[r] & K_1(S\phi) \arrow[r] & K_0(\phi\otimes\id_\K) \arrow[r] & K_0(\phi\otimes\id_{\T_0})
\end{tikzcd}
	\end{center}
	is exact, and Lemma \ref{nuclear} and Lemma \ref{stable} together give an isomorphism $K_0(\phi)\cong K_1(S\phi)$. We let $\beta_\phi:K_0(\phi)\to K_1(S\phi)$ denote this isomorphism. We introduce a useful piece of notation before giving an explicit description of $\beta_\phi$.
	
	\begin{definition}\label{projpath}
	For a triple $(p,1_n,v)$ in $\Gamma_0(\phi)$, choose $m\geq n$ such that $p$ is in $M_m(\tilde A)$, and let
	\[p_v(s)=w(s)^*\left[\begin{array}{cc}
	    1_n & 0 \\
	    0 & 0_{2m-n}
	\end{array}\right]w(s)\]
	where $w$ is a path of unitaries in $M_{2m}(\tilde B)$ with $w(0)=1_{2m}$ and 
	\[w(1)=\left[\begin{array}{cc}
	    v & 1_m-vv^* \\
	    1_m-v^*v & v^*
	\end{array}\right]\]
	\end{definition}
	Note that such a path $w$ exists since
	\[\left[\begin{array}{cc}
	    v & 1_m-vv^* \\
	    1_m-v^*v & v^*
	\end{array}\right]=\left[\begin{array}{cc}
	    0 & 1_m \\
	    1_m & 0
	\end{array}\right]\left[\begin{array}{cc}
	    1_m-v^*v & v^* \\
	    v & 1_m-vv^*
	\end{array}\right]\]
	and the two unitaries on the right are self-adjoint.
	
	We then have $\beta_\phi([p,1_n,v])=[1_{2m},u,g]$ where
	\[u(t)=\exp\left(2\pi it\left[\begin{array}{cc}
	    p & 0 \\
	    0 & 0_m
	\end{array}\right]\right)\exp\left(-2\pi it\left[\begin{array}{cc}
	    1_n & 0 \\
	    0 & 0_{2m-n}
	\end{array}\right]\right)\]
	and
	\[g(s,t)=\exp(2\pi itp_v(s))\exp\left(-2\pi it\left[\begin{array}{cc}
	    1_n & 0 \\
	    0 & 0_{2m-n}
	\end{array}\right]\right)\]
	
	Now we complete the six-term exact sequence in part (iii) of Theorem \ref{main}. We define the exponential map $\partial_0:K_0(\gamma)\to K_1(\psi)$ to be the group homomorphism that makes the diagram
	\begin{center}
\begin{tikzcd}
K_0(\gamma) \arrow[rr, "\partial_0", dashed] \arrow[dd, "\beta_\gamma"] &  & K_1(\psi) \arrow[dd, "\theta_\psi"] \\
                                                                        &  &                                     \\
K_1(S\gamma) \arrow[rr, "\partial_2"]                                   &  & K_0(S\psi)                         
\end{tikzcd}
	\end{center}
	commutative. All maps in the above diagram are natural, so the sequence in part (iii) of Theorem \ref{main} is exact everywhere.
	
	An explicit description of $\partial_0$ is as follows. Given a triple $(p,1_n,v)$ in $\Gamma_0(\gamma)$, choose $m$ and $p_v$ as in Definition \ref{projpath}. Let $a$ in $M_m(\tilde A)$ be such that $a=a^*$, $\pi_A(a)=p$, and let $f$ in $M_{2m}(\widetilde{CB})$ be such that $f(t)=f(t)^*$ for all $t$, $\pi_B(f)=p_v$, and $f(1)=\phi(a)\oplus0_m$. Then we have
	\[\partial_0([p,1_n,v])=-[1_{2m},\exp(2\pi i(a\oplus0_m)),\exp(2\pi if)]\]
	
	\begin{remark}
	 It is interesting to note that split exactness was not necessary to prove part (iii) of Theorem \ref{main}, since split exactness is crucial to deduce that $K_*(\T_0)=0$ from the isomorphism $K_*(\T)\cong K_*(\bC)$ during the proof of Bott periodicity in \cite{cuntz}. Here, we were able to sneak around this difficulty using Corollary \ref{zero} and the fact that $K_*(\T_0)=0$.
	\end{remark}
	
	\subsection{Proof of part (iv) of Theorem \ref{main}}

Consider the commutative diagram

\begin{center}
\begin{tikzcd}
0 \arrow[r] & SB \arrow[r, "\iota_A"] \arrow[d, equal] & C_\phi \arrow[r, "\pi_A"] \arrow[d, "\sigma"] & A \arrow[r] \arrow[d, "\phi"] & 0 \\
0 \arrow[r] & SB \arrow[r, hook]           & CB \arrow[r, "\pi_B"]                       & B \arrow[r]                   & 0
\end{tikzcd}
\end{center}
with exact rows, where $\iota_A(f)=(0,f)$, $\pi_A(a,f)=a$, $\pi_B(f)=f(1)$, and $\sigma(a,f)=f$. Since $K_*(CB)=0$ because $CB$ is contractible, we have by part (i) of Theorem \ref{exc} that $\nu_j:K_j(\sigma)\to K_j(C_\phi)$ is an isomorphism for $j=0,1$. By part (iii) of Theorem \ref{main} and Corollary \ref{zero},

\begin{center}
\begin{tikzcd}
0 \arrow[rr]         &  & K_0(\sigma) \arrow[rr, "\pi_*"] &  & K_0(\phi) \arrow[dd] \\
                     &  &                      &  &                      \\
K_1(\phi) \arrow[uu] &  & K_1(\sigma) \arrow[ll, "\pi_*"'] &  & 0 \arrow[ll]        
\end{tikzcd}
\end{center}
	is exact and hence $\pi_*:K_j(\sigma)\to K_j(\phi)$ is an isomorphism for $j=0,1$.
	
	\begin{definition}
	$\Delta_\phi=\nu_j\circ\pi_*^{-1}$, for $j=0,1$.
	\end{definition}
	
	We provide a description of $\Delta_\phi$. For simplicity, we assume that $A$ and $B$ are unital and that $\phi(1)=1$. Given a triple $(p,1_n,v)$ in $\Gamma_0(\phi)$, we have 
	\[\Delta_\phi([p,1_n,v])=[(p\oplus0_m,p_v)]-[1_n\oplus0_{2m-n}]\]
	where $m$ and $p_v$ are as in Definition \ref{projpath}. Given a triple $(1_n,u,g)$ in $\Gamma_1(\phi)$, we have
	\[\Delta_\phi([1_n,u,g])=[(u,g)].\]
	The proof that the diagram given in part (iv) of Theorem \ref{main} is commutative is straightforward and is left to the reader.
 
\subsection{Proof of Theorem \ref{axiom}} We have already proven part (ii) (Lemma \ref{stable}), so it remains to prove parts (i) and (iii). Both proofs are quite easy with the natural transformation $\Delta$ from part (iv) of Theorem \ref{main} in hand.

The assumptions on $\alpha_t$ and $\beta_t$ in part (i) clearly imply that $\alpha_t\oplus C\beta_t$ is a continuous path of $^*$-homomorphisms from $C_\phi$ to $C_\psi$. It follows from homotopy invariance of C$^*$-algebra $K$-theory that $(\alpha_0\oplus C\beta_0)_*=(\alpha_1\oplus C\beta_1)_*$, and hence that $(\alpha_0,\beta_0)_*=(\alpha_1,\beta_1)_*$ via the natural isomorphism $\Delta$.

For part (iii), the existence of $\phi:A\to B$ is an easy consequence of the universal property of inductive limits, using the $^*$-homomorphisms $\nu_i\circ\phi_i$. It is clear that $(C_{\phi_i},\alpha_{ij}\oplus C\beta_{ij})$ forms an inductive system of C$^*$-algebras, and the limit is $(C_\phi,(\mu_i\oplus C\nu_i))$ by Proposition 4.9 of \cite{pedersen}. The result then follows from continuity of C$^*$-algebra $K$-theory and the natural isomorphism $\Delta$.

\subsection{Proof of Theorem \ref{exc}} Parts (i) and (ii) follow immediately from exactness. Part (iii) follows from applying part (iii) of Theorem \ref{main} to the diagram
\begin{center}
    \begin{tikzcd}
0 \arrow[r] & \ker\phi \arrow[r, "\iota_\phi"] \arrow[d] & A \arrow[r, "\phi"] \arrow[d, "\phi"] & B \arrow[r] \arrow[d, equal] & 0 \\
0 \arrow[r] & 0 \arrow[r, hook]           & B \arrow[r, equal]                       & B \arrow[r]                   & 0
\end{tikzcd}
\end{center}
Part (iv) follows similarly, using the diagram
\begin{center}
    \begin{tikzcd}
0 \arrow[r] & I \arrow[r, hook] \arrow[d, hook] & A \arrow[r, "\pi_A"] \arrow[d, equal] & A/I \arrow[r] \arrow[d] & 0 \\
0 \arrow[r] & A \arrow[r, equal]           & A \arrow[r]                       & 0 \arrow[r]                   & 0
\end{tikzcd}
\end{center}




\end{document}